\documentclass[10pt,a4paper]{article}
\usepackage[english]{babel}
\usepackage{amssymb}
\usepackage{amsmath}
\usepackage{indentfirst}
\usepackage[dvips]{graphicx}
\usepackage{epsfig}
\usepackage{xcolor}

\usepackage[colorlinks,citecolor=blue,linkcolor=red]{hyperref}

\usepackage{color}

\def\cambA{\marginpar{\textcolor{magenta}{changeAnna}}\textcolor{magenta}}

\usepackage{amsthm}
\usepackage{comment}
\usepackage{multirow}  
\usepackage{amsfonts}
\usepackage{mathtools} 
\usepackage{setspace} 

\newtheorem{defi}{Definition}[section]
\newtheorem{cor}{Corollary}[section]

\newtheorem{lemma}[cor]{Lemma}
\newtheorem{teo}[cor]{Theorem}

\newtheorem{rem}[cor]{Remark}

\newtheorem{compa}[cor]{Compatibility conditions}

\newcommand{\R}{\mathbb{R}}
\newcommand{\N}{\mathbb{N}}

\DeclareMathOperator{\dimension}{dim}
\DeclareMathOperator{\Span}{span}

\newcommand{\op}{\operatorname}

\newcommand{\col}{\coloneqq}
\newcommand{\pa}{\partial}

\numberwithin{equation}{section}

\begin{document}

\title{Elastic flow of networks: short-time existence result}

\author{{\sc Anna Dall'Acqua} \thanks{Institut f\"ur Analysis, Universit\"at Ulm, Germany,
\url{anna.dallacqua@uni-ulm.de}}
{\sc Chun-Chi Lin} \thanks{Department of Mathematics, National Taiwan Normal University, Taipei, 116 Taiwan, \url{chunlin@math.ntnu.edu.tw}} and
 {\sc Paola Pozzi}  \thanks{Fakult\"at f\"ur Mathematik, Universit\"at Duisburg-Essen, Germany, \url{paola.pozzi@uni-due.de}} 
}

\date{\today}
\maketitle

\begin{abstract}
In this paper we study the $L^2$-gradient flow of the penalized elastic energy on networks of $q$-curves in $\R^{n}$ for 
$q \geq 3$. Each curve is fixed at one end-point and at the other is joint to the other  curves at a movable $q$-junction. 
For this geometric evolution problem with natural boundary condition  
we show the existence of smooth solutions for a (possibly) short interval of time. 
Since the geometric problem is not well-posed, due to the freedom in reparametrization of curves, 
we consider a fourth-order non-degenerate parabolic quasilinear system, called the analytic problem,  and show first a  short-time existence result for this parabolic system.  
The proof relies on applying Solonnikov's theory on linear parabolic systems 
and Banach fixed point theorem in proper H\"{o}lder spaces. 
Then the original geometric problem is solved by establishing the relation between the analytical solutions and the solutions to the geometrical problem. 
\end{abstract}


\bigskip


\tableofcontents

\section{Introduction}

The elastic energy of a smooth regular curve immersed in $\mathbb{R}^{n}$, 
$f: \bar{I} \to \mathbb{R}^{n}$, $n \geq 2 $, $I=(0,1)$, 
is given by 
\begin{equation}\label{Elasten}
\mathcal{E}(f)=\frac{1}{2} \int_I |\vec{\kappa}|^2 ds,
\end{equation}
where $ds=|\partial_x f| dx$ is the arc-length element and $\vec{\kappa}$ is the curvature vector of the curve. 
The latter is given by $\vec{\kappa}=\partial_s^2 f$ where $\partial_s = |\partial_x f|^{-1} \partial_x $ denotes the differentiation with respect to the arclength parameter. 
The elastic energy in (\ref{Elasten}) is also called bending energy of curves.  
It was proposed by Jacob Bernoulli in 1691 for studying the equilibrium shape of curves,
called elasticae or elastic curves, \cite{Truesdell}. 
Besides being used as a simple model in mechanics, 
the elastic energy has also been used for defining and studying the so-called nonlinear splines in computer graphics, 
see e.g., \cite{LF} and the references therein. 

Since the elastic energy of a curve can be made arbitrarily small by enlarging the curve, in minimization problems one usually penalizes the length or consider curves with fixed length. In the first case one is led to consider the energy 
\begin{equation}\label{Elambda}
\mathcal{E}_{\lambda}(f)=\mathcal{E}(f) + \lambda \mathcal{L}(f) \, , \quad \lambda > 0 .
\end{equation}
where 
\begin{equation*}
\mathcal{L}(f)=\int_I  ds 
\end{equation*}
is the length of the curve. The term $\lambda \mathcal{L}(f)$, when $\lambda>0$, in \eqref{Elambda} is a natural term to be considered, 
since it could be viewed as the energy naively responsible for the stretching of curves in elasticity.

In both cases, critical points of the energy satisfy the equation
$$\nabla_{L^2} \mathcal{E}_{\lambda}(f) = \nabla_s^2 \vec{\kappa}+\frac12 |\vec{\kappa}|^2 \vec{\kappa} - \lambda \vec{\kappa} =0 \, ,$$
where in the case of fixed length $\lambda$ is a Lagrange multiplier (see for istance \cite{DKS}, \cite{Polden}). 
Here $\nabla_s$ is an operator that on a smooth vector field $\phi$ acts as follows $\nabla_s \phi = \partial_s \phi - \langle \partial_s \phi, \partial_s f \rangle \partial_s f$, i.e., it is the normal projection of $\partial_s \phi$. It may also be understood as a covariant differentiation.

The attempt to associate the elastic energy to networks appears in some investigation of polymer gels, fiber or protein networks in mechanical engineering or material sciences (e.g., see \cite{BC07}, \cite{GD14}). The mathematical treatment of networks with elastic energy has  started quite recently. In \cite{DNP,NPP} the authors provide first results concerning the existence of minimizers in  special classes of networks (in particular an angle condition is imposed at the junction). Here we look at the steepest descent flow of the elastic energy on networks of $q$ curves, $q \geq 3$, starting from a point (the $q$-junction) and ending at $q$ 
fixed points in $\R^n$. 
In this setting we assign orientation to each curve, although the energy is independent of the orientation of the curves. 
In other words, the network $f=\{f_1,f_2, \dots,f_q\}$, 
where $f_i:\bar{I} \to \R^n$, $I=(0,1)$, $i\in \{1, \dots,q\}$, are $q$ regular curves, 
satisfies the followings: 
\begin{enumerate}
\item The end-points are fixed: 
\begin{equation}\label{IPs}
f_i(1)=P_i \mbox{ for } i\in \{1, \dots,q\},
\end{equation}
with given  points $P_i$, $i\in \{1, \dots,q\}$, in $\R^n$. 
\item The curves start at the same point
\begin{equation}\label{JPs}
f_i(0)=f_j(0) \mbox{ for all } i,j\in \{1, \dots,q\} \,  \quad (\text{concurrecy condition}). 
\end{equation}
\end{enumerate}
We write $\Gamma=\{f_1,f_2, \dots,f_q\}$ when we think of the network as a geometrical object, that is when the parametrization chosen for each curve plays no role. The energy of the network $\Gamma=\{f_1,\dots, f_q\}$ is given by
\begin{equation}\label{energyNetwork}
\mathcal{E}_{\lambda}(\Gamma) = \sum_{i=1}^q \mathcal{E}_{\lambda_{i}}(f_i) \, ,
\end{equation}
where $\lambda=(\lambda_1,\dots,\lambda_q)$, $\lambda_i \geq 0$ (the penalization of the length is obviously not necessary for a short time existence result).

We call the above configuration a $q$-network. The aim of this work is to complete our  work undertaken in \cite{DLPnetwork1}, where we analyse the long time behaviour  for the elastic flow of triods (3-networks). More precisely,  we give here full details on the short-time existence result exploited in \cite{DLPnetwork1}. At the same time we generalize the needed short-time existence statement to the case of $q$-networks for $q \geq 3$. 
A short-time existence result for the elastic evolution of networks appeared first in \cite{GMP}: there the planar case for triods is discussed and the existence is demonstrated in $C^{\frac{4+\alpha}{4}, 4+\alpha}$ spaces. For our arguments in \cite{DLPnetwork1} to be complete we need however a statement for networks  with curves in $\R^{n}$ whose parametrization is smooth is space and time up to time zero. As it turns out, we are able to demonstrate what is needed, independently of the number ($q \geq 3$) of curves meeting at the junction.

For an overview on the current research undertaken on the elastic flow of networks, we refer the interested reader to \cite{DLPnetwork1, NP19, GMPLTE, BGN12b} and the references given there.

\subsection{Main results}

The aim of this work is to establish a short time existence result for the $L^2$-gradient flow of the energy $\mathcal{E}_{\lambda}$ of a network as described above. In other words, 
given an initial $q$-network $\Gamma_0=\{f_{0,1},\dots, f_{0,q}\}$ of sufficiently smooth regular   curves satisfying \eqref{IPs} and \eqref{JPs}, 
we look for the existence of $T>0$ and  $f_i:[0,T]\times [0,1] \to \R^n$, $f_{i} \in C^{\frac{k+\alpha}{4}, k+\alpha}([0,T]\times [0,1]) $,  $k \in \mathbb{N}$, $k \geq 4$, $\alpha \in (0,1)$ (resp. $f_{i} \in C^{\infty} ([0,T]\times [0,1])$, see Appendix \ref{sec:fs} for the definition of the parabolic H\"older spaces) for $i\in \{1, \dots,q\}$, regular curves and solution to
\begin{equation*}
(\partial_t f_i)^{\perp}= - \nabla_s^2 \vec{\kappa_i}-\frac12 |\vec{\kappa}_i|^2 \vec{\kappa}_i + \lambda_{i} \vec{\kappa}_i, \quad i\in \{1, \dots,q\},
\end{equation*}
with initial datum $f_i(t=0)=f_{0,i}$
 and boundary conditions
\begin{equation}\label{eq:nonlinearbcgeo}
\left\{\begin{aligned}
f_i(t,1)& =P_i, &\mbox{ for all }t\in [0,T], i\in \{1, \dots,q\}, \\
\vec{\kappa}_i(t,1) & =0=\vec{\kappa}_i(t,0)  &\mbox{ for all }t\in [0,T], \ i\in \{1, \dots,q\},\\
f_i(t,0) & =f_j(t,0) &\mbox{ for all }t\in [0,T],  \ i,j\in \{1, \dots,q\},\\
\mbox{and }& \sum_{i=1}^q (\nabla_s \vec{\kappa}_{i}(t,0) - \lambda_i \partial_s f_{i}(t,0)) = 0  &\mbox{ for all }t\in [0,T].
\end{aligned} \right.
\end{equation}
As usual $(\pa_t f)^{\perp}$ denotes the normal part of the velocity, i.e. $(\pa_t f)^{\perp} = \pa_t f -\langle \pa_t f, \pa_s f \rangle \pa_s f$. The first and third line in \eqref{eq:nonlinearbcgeo} ensure that during the flow the network satisfies \eqref{IPs} and \eqref{JPs}, while the other boundary conditions are the so called natural ones, derived by imposing that the first variation of the energy is zero. For the derivation of the first variation and the natural boundary conditions in the case $q=3$, the readers are referred to Section 2 of \cite{DLPnetwork1} (see also Appendix~\ref{secA} below). The case of general $q$ goes similarly.

For the initial datum 
$\Gamma_0=\{f_{0,1}, \dots, f_{0,q}\}$, we assume that $f_{0,i} \in C^{k, \alpha}([0,1], \R^{n})$, $k \geq4$, $\alpha \in (0,1)$ (resp. $f_{0,i} \in C^{\infty}([0,1], \R^{n})$),   $ i\in \{1, \dots,q\}$, are regular curves such that at the boundary points
\begin{equation}\label{eq:nonlinearbcgeo2}
\left\{\begin{aligned}
 f_{0,i}(1) &=P_i, & \mbox{ for all } i\in \{1, \dots,q\}, \\
\vec{\kappa}_{0,i}(1) & =0=\vec{\kappa}_{0,i}(0),  &\mbox{ for all }  i\in \{1, \dots,q\},\\
f_{0,i}(0) & =f_{0,j}(0)  & \mbox{ for all }  i,j\in \{1, \dots,q\},\\
\mbox{and }& \sum_{i=1}^q (\nabla_s \vec{\kappa}_{0,i}(0) - \lambda_i \partial_s f_{0,i}(0)) = 0 \, ,
\end{aligned} \right.
\end{equation}
(where $\vec{\kappa}_{0,i}$ denotes the curvature of $f_{0,i}$) and with further compatibility conditions (specified in the 
statements below).
 Furthermore, the initial datum has to satisfy the following non-collinearity condition.

\begin{defi}[Non-collinearity condition \textbf{(NC)}] We say that the initial datum satisfies the non-collinearity condition if
\begin{equation*}
\left. \dimension \Span\{\partial_s f_{0,1},\dots,\partial_s f_{0,q}\}\right|_{x=0}\geq 2 \, .
\end{equation*} 
Similarly, a family of regular curves $f_i:[0,T]\times [0,1] \to \R^n$, $i\in \{1, \dots,q\},$ , 
satisfies the non-collinearity condition if 
\begin{equation*}
\left. \dimension \Span\{\partial_s f_{1}(t,x),\dots,\partial_s f_{q}(t,x)\}\right|_{x=0}\geq 2 \mbox{ for all }t\in [0,T]\, .
\end{equation*} 
\end{defi}
\begin{rem}\label{rem:1.1}
The non-collinearity condition (NC)  
establishes that the $q$ unit tangent vectors at the $q$-junction should not span a one-dimensional subspace. Analytically and equivalently, we can express this fact by considering the (geometric) expression $nc:[0,T]\times [0,1] \to \R$,
$$nc(t,x)=1 -   \prod_{1\leq i<j\leq q} \Big|\langle \partial_{s} f_{i}(t,x), \partial_{s} f_{j}(t,x) \rangle\Big| \, ,$$
and asking that $nc$ is strictly positive at the junction point $x = 0 $.

As we will see below, the non-collinearity condition is necessary in our analysis in order to guarantee the short time existence of a solution. Moreover it has been used also in \cite{DLPnetwork1} to prove long-time existence (in the case $q=3$). Note that in case $q=3$, then  $nc$ is simply given by
$$nc= 1- \langle \partial_{s} f_{1}(t,x), \partial_{s} f_{2}(t,x) \rangle \langle \partial_{s} f_{1}(t,x), \partial_{s} f_{3}(t,x) \rangle \langle \partial_{s} f_{2}(t,x), \partial_{s} f_{3}(t,x) \rangle.$$ 
Indeed, in \cite[\S~5]{DLPnetwork1} it is shown that the non-collinearity condition arises naturally when imposing
$\partial_{t} f_{i}=\partial_{t}f_{j}$ at the junction: in particular if $nc>0$ holds then \emph{at the boundary} the tangential components 
 of the velocity vectors (that is $\langle \partial_{t} f_{i}, \partial_{s} f_{i} \rangle$) 
can be expressed in purely geometric terms. See Remark \ref{rem:4.1}  below for the arguments and the generalization to the case of $q$ curves. 
\end{rem}
Observe that the formulation of the problem  given so far involves purely geometric quantities and hence it is invariant under reparametrizations.

In order to treat the problem analytically and 
to keep the topology of the network with movable junction point during the evolution, 
we need to allow some tangential components in the flow equations. Hence, we rewrite the flow equations as
\begin{equation}\label{eq:flowgeomtang}
\partial_t f_i  = - \nabla_{s}^2 \vec{\kappa}_i -\frac12 |\vec{\kappa}_i|^2 \vec{\kappa}_i  + \lambda_{i} \vec{\kappa}_i + \varphi_i \partial_s f_i \mbox{ on } (0,T) \times I \mbox{ for }  i\in \{1, \dots,q\},
\end{equation}
for some (sufficiently smooth, that is $\varphi_{i} \in C^{{\frac{k+\alpha-4}{4}},k+\alpha-4} ([0,T]\times [0,1]) $ resp. $\varphi_{i} \in C^{\infty} ([0,T]\times [0,1])$) tangential components $\varphi_{i}=\langle \partial_{t} f_{i}, \partial_{s} f_{i} \rangle$, which are part of the problem.\\
Our main result reads as follows.

\begin{teo}[Geometric existence Theorem]\label{teo:STEgeo}
Let $n \geq 2$, $q\geq 3$, 
$\alpha\in(0,1)$ and $P_i$, $i\in \{1, \dots,q\}$, be given points in $\R^n$.  
Given $f_{0,i}:[0,1] \to \R^n$, $f_{0,i} \in C^{4,\alpha}([0,1])$, $i\in \{1, \dots,q\}$, regular curves, satisfying the non-collinearity condition (NC), \eqref{eq:nonlinearbcgeo2}, and 
\begin{align}\label{eq:ccfirststep-zero}
\nabla_s^2 \vec{\kappa}_{0,i}&=0 \qquad \mbox{ at }x=1, \quad  i\in \{1, \dots,q\},  \\ \label{eq:ccfirststep}
-\nabla_s^2 \vec{\kappa}_{0,i} + \varphi_{0,i} \partial_{s}f_{0,i}&=-\nabla_s^2 \vec{\kappa}_{0,j}  +\varphi_{0,j} \partial_{s}f_{0,j}
\quad \mbox{ at }x=0 \quad \mbox{ for }i, j \in \{1, \dots,q\} , 
\end{align}
with $\varphi_{0,i}$ defined in \eqref{eq:defvarphi0} below, then there exist $T>0$ and 
regular curves $f_{i} \in C^{\frac{4+\alpha}{4},4+\alpha}([0,T]\times I;\R^n)$, $i\in \{1, \dots,q\}$, 
 such that 
\begin{equation*}
(\partial_t f_i)^{\perp}= - \nabla_s^2 \vec{\kappa_i}-\frac12 |\vec{\kappa}_i|^2 \vec{\kappa}_i + \lambda_{i} \vec{\kappa}_i, \quad i\in \{1, \dots,q\},
\end{equation*}
together with the boundary conditions \eqref{eq:nonlinearbcgeo} 
and the initial condition $\Gamma=\{f_1, \dots,  f_q\}|_{t=0}$ equal to $\Gamma_{0}=\{f_{0,1}, \dots, f_{0,q}\}$  
that is
\begin{equation}\label{eq:icgeo}
f_i (t=0)=f_{0,i}\circ \phi_{i} , \quad i\in \{1, \dots,q\},
\end{equation}
with $\phi_{i}\in C^{4, \alpha}([0,1], [0,1])$, orientation preserving diffeomorphisms. Moreover, we have instant parabolic smoothing, that is  $f_i \in C^{\infty}((0,T]\times [0,1])$ for any $i\in \{1, \dots,q\}$, and the non-collinearity condition holds at the triple junction for any time $t \in [0,T]$.
\end{teo}
It turns out that also \eqref{eq:ccfirststep} is a fully geometric condition, as discussed in detail in Remark~\ref{rem:4.1} (cf. also Remark~\ref{rem:1.1} above). Since the problem and formulation are fully geometric, it is natural and consistent that in \eqref{eq:icgeo}  we should not fix the parametrization of the initial data.

Upon imposing higher regularity and stronger compatibility condition for the initial data, we can obtain a smooth solution. Precisely

\begin{teo}[Smooth Geometric existence Theorem]\label{teo:STEgeoCinfty}
Let $n \geq 2$, $q\geq 3$, 
and $P_i$, $i\in \{1, \dots,q\}$, be given points in $\R^n$.  
Given $f_{0,i}:[0,1] \to \R^n$, $f_{0,i} \in C^{\infty}([0,1])$, $i\in \{1, \dots,q\}$, regular curves 
which, when parametrized by constant speed, satisfy the compatibility conditions of any order (as stated in Remark~\ref{rem2.1} below) and the non-collinearity condition (NC), then there exist $T>0$ and 
regular curves $f_{i} \in C^{\infty}([0,T]\times [0,1];\R^n)$, $i\in \{1, \dots,q\}$, 
 such that 
 \begin{equation*}
(\partial_t f_i)^{\perp}= - \nabla_s^2 \vec{\kappa_i}-\frac12 |\vec{\kappa}_i|^2 \vec{\kappa}_i + \lambda_{i} \vec{\kappa}_i, \quad i\in \{1, \dots,q\},
\end{equation*}
together with the boundary conditions \eqref{eq:nonlinearbcgeo} and the initial condition 
$\Gamma=\{f_1, \dots , f_q\}|_{t=0}$ equal to $\Gamma_{0}=\{f_{0,1}, \dots, f_{0,q}\}$ (in the sense of \eqref{eq:icgeo} for smooth orientation preserving diffeomorphisms).
Moreover, the non-collinearity condition (NC) holds at the $q$-junction for any time $t \in [0,T]$.
\end{teo}

\begin{rem} In order to be able to use the expression ``network'' to describe our geometrical setting we restrict ourself to the case of $q$ curves  with $q \geq 3$, but as the analysis below shows, all arguments used are still valid also in the case of $q=2$. Moreover, the analysis performed below can be easily generalized to  $q$-networks with curves meeting in two $q$-junctions (the so called theta-networks when  $q=3$), therefore the previous results hold also for this configuration.
\end{rem}
The problem of geometric uniqueness is briefly treated in Lemma~\ref{geomuniq}.

\begin{rem}\label{rem:4.1}
In the statement of Theorem \ref{teo:STEgeo} we ask that the initial datum satisfies \eqref{eq:ccfirststep}. This condition is geometrical (i.e. independent of the choice of parametrization) since  the non-collinearity condition (NC) holds along the flow. This has been observed and exploited already in \cite[Rem.5.1]{DLPnetwork1} in the case of $q=3$. In the case of general $q$ the argument goes as follows. If $f_i,\varphi_i$, $i=1,2,\ldots,q$, solve \eqref{eq:flowgeomtang}, \eqref{eq:nonlinearbcgeo}, \eqref{eq:nonlinearbcgeo2} in the sense of Theorem \ref{teo:STEgeo} (in particular also the compatibility conditions of order zero are satisfied), then at $x=0$ we have that for any $t\in [0,T)$,  $\partial_{t} f_{i}= \partial_{t} f_{j}$, that is
\begin{align*}
-A_{i} + \varphi_{i} T_{i} = -A_{j} + \varphi_{j} T_{j}
\end{align*}
for any $i,j \in \{ 1, \ldots,q \}$, where for brevity of notation we write
$$A_{i}= A_{i}(t):=\nabla_{s}^2 \vec{\kappa}_i \Big|_{x=0}, \qquad  T_{i}=T_{i}(t)= \partial_{s} f_{i}(t,0),$$
and where we have used the fact that the curvature vanishes at the boundary.
Taking the scalar product with $T_{i}$ gives  $\varphi_{i}= - \langle A_{j}, T_{i} \rangle + \varphi_{j} \langle T_{j}, T_{i} \rangle$, and summing up yields
\begin{align*}
(q-1) \varphi_{i} = \sum_{ j \neq i,\, j=1}^{q} \varphi_{i}=\sum_{ j \neq i,\, j=1}^{q} (- \langle A_{j}, T_{i} \rangle  + \varphi_{j} \langle T_{j}, T_{i} \rangle).
\end{align*}
In other words, for any $i \in \{ 1, \ldots, q \}$ we have
\begin{align*}
(q-1) \varphi_{i} - \sum_{ j \neq i,\, j=1}^{q}\varphi_{j} \langle T_{j}, T_{i} \rangle=    \sum_{ j \neq i,\, j=1}^{q}- \langle A_{j}, T_{i} \rangle,
\end{align*}
which can be written as
\begin{align}\label{phi1in0}
Q\cdot\left(
\begin{array}{c}
\varphi_{1}\\\varphi_{2} \\ \vdots \\ \varphi_{q}
\end{array} \right) = \left(
\begin{array}{c}
- \langle  \sum_{ j \neq 1,\, j=1}^{q} A_{j}, T_{1} \rangle\\
- \langle  \sum_{ j \neq 2,\, j=1}^{q} A_{j}, T_{2} \rangle\\
\vdots \\
- \langle  \sum_{ j \neq q,\, j=1}^{q} A_{j}, T_{q} \rangle\\
\end{array} \right),
\end{align}
where
\begin{align*}
Q=\left(
\begin{array}{ccccc}
(q-1) & - \langle T_{1}, T_{2} \rangle& - \langle T_{1}, T_{3} \rangle & \ldots &  -\langle T_{1}, T_{q} \rangle\\
- \langle T_{2}, T_{1} \rangle & (q-1) &- \langle T_{2}, T_{3} \rangle & \ldots & -\langle T_{2}, T_{q} \rangle\\
\vdots & & & & \vdots\\
- \langle T_{q}, T_{1} \rangle& - \langle T_{q}, T_{2} \rangle &- \langle T_{q}, T_{3} \rangle & \ldots& (q-1)
\end{array} \right)
\text{.} 
\end{align*} 
We see that the submatrix $(Q_{ij})_{i,j=1}^{q-1} \in \R^{(q-1)\times (q-1)}$ 
composed out of the first $(q-1)$ rows and $(q-1)$ columns is strictly diagonal dominant and hence invertible.
This in turns implies that the first $(q-1)$ columns of $Q$ are linearly independent and hence $rank\, (Q) \geq q-1$. The rank of $Q$ is precisely $q-1$ if we can write the last column as a linear combination of the first $(q-1)$, that is if $Q$ has zero as an eigenvalue. 
Hence, suppose there exists $v \in \R^{q}$, $v \neq0$, such that $Qv=0$. Let $w=\frac{v}{\|v\|_{\infty}}$. Then $Qw=0$ and $|w_{i}| \leq 1$ for any $i=1, \ldots q$. Possibly multiplying $w$ with $-1$ we obtain the existence of an entry $w_{j}$ such that $w_{j}=1$. From $(Qw)_{j}=0$ it follows then
$$ (q-1)= \sum_{i\neq j, i=1}^{q} w_{i}\langle T_{i}, T_{j} \rangle$$
Since $|w_{i}\langle T_{i}, T_{j} \rangle| \leq |\langle T_{i}, T_{j} \rangle| \leq 1$ this can not be realized if $dim \{ T_{1}, \ldots, T_{q} \} \geq2.$
Hence it becomes clear that the validity of the non-collinearity condition (NC) at time $t \in [0,T)$ ensures the invertibility of $Q$ and thus also the fact that $\varphi_i(t,0)$, $i=1, \ldots, q$, $t \in [0,T)$ can be expressed in geometrical terms at the junction point, namely
\begin{equation}\label{eq:defvarphi0}
\varphi_{i}(t,x=0) = \varphi_{i}\big( \partial_{s} f_{j}(t,0),\nabla_{s}^2 \vec{\kappa}_{j}(t,0), j=1,2,\ldots, q \big)
\end{equation}
where the exact expression can be immediately deduced from \eqref{phi1in0}.

\end{rem}

\subsection{Structure of the article}

In the next section we give the analytical problem which we are going to solve. 
We consider \eqref{eq:flowgeomtang} with a specific choice of the tangential component $\varphi_{i}$ (cf. \eqref{varphi*} below) and with boundary conditions a bit stronger than \eqref{eq:nonlinearbcgeo}. These choices 
are dictated by the need to obtain a (fourth order) non-degenerate system of quasilinear PDEs. 
In Section~\ref{sec:PT}, we give the proof of the short-time existence for the non-degenerate parabolic system of fourth-order, solving the analytic problem, 
by applying Solonnikov's theory on linear parabolic systems and Banach fixed point theorem in proper 
H\"{o}lder spaces. 
In Section~\ref{sec:equiv} we discuss  the relation between the analytical solutions obtained in Section~\ref{sec:ap} and the solutions to the geometrical problem we are interested in.
In the appendix we collect some useful results.

\bigskip


\noindent \textbf{Acknowledgements:} This project has been funded by the Deutsche Forschungsgemeinschaft (DFG, German Research Foundation)- Projektnummer: 404870139, 
and Ministry of Science and Technology, Taiwan (MoST 107-2923-M-003 -001 -MY3).

\section{The analytical problem}
\label{sec:ap}

As already observed, the geometrical problem we want to study is not well posed due to the freedom given by the invariance with respect to reparametrizations. This is why, we consider now a fourth order (non-degenerate) system of quasilinear PDEs for which we prove existence of a solution.

In the flow equations \eqref{eq:flowgeomtang} 
only the normal components of derivatives with respect to arc-length appear and hence the operator is not uniformly elliptic. From formula \eqref{eq:Aflowexpl} we see that 
by choosing tangential components 
\begin{align}\label{varphi*}
\varphi^{*}_{i}&=
 -  \langle \frac{\partial_{x}^{4}f_{i}}{|\partial_{x}f_{i}|^{5}}, \partial_{x}f_{i} \rangle + 10  \frac{\langle \partial_{x}^{2}f_{i}, \partial_{x}f_{i} \rangle}{|\partial_{x}f_{i}|^{7}} \langle \partial_{x}^{3}f_{i}, \partial_{x}f_{i} \rangle  +\frac{5}{2} \langle \partial_{x}^{2}f_{i}, \partial_{x}f_{i} \rangle \frac{|\partial_{x}^{2}f_{i}|^{2}}{|\partial_{x}f_{i}|^{7}} \\ 
& \quad 
-\frac{35}{2}\frac{(\langle \partial_{x}^{2}f_{i}, \partial _{x}f_{i} \rangle)^{3}}{|\partial_{x} f_{i}|^{9}} 
+ \lambda_{i} \frac{\langle \partial_{x}^{2}f_{i}, \partial_{x} f_{i} \rangle}{|\partial_{x}f_{i}|^{3}}  \notag
\end{align}
we get the parabolic equations
\begin{equation} \label{(P)}
\pa_t f_i=-\frac{1}{|\pa_x f_i|^4}\pa_{x}^4 f_i+h(f_i),
\end{equation}
for $i=1, \dots,q$, with 
\begin{align}\label{eq:gi}
h (f_i) & =  6 \langle \partial_x^2 f_{i}, \partial_x f_{i} \rangle \frac{ \partial_x^3 f_{i}}{|\partial_x f_{i}|^{6}}  \\ \nonumber
&  \quad+
\frac{ \partial_x^2 f_{i}}{|\partial_x f_{i}|^{2}} \Big{(} \frac{5}{2} \frac{|\partial_x^2 f_{i}|^{2}}{|\partial_x f_{i}|^{4}}  
+ 4 \frac{\langle \partial_x^3 f_i, \partial_x f_i \rangle}{|\partial_x f_{i}|^{4}} - \frac{35}{2} \frac{(\langle \pa_x^2 f_{i}, \pa_x f_{i} \rangle)^{2}}{|\pa_x f_{i}|^{6}} +\lambda_{i}\Big{)}\, . 
\end{align}
Further, we observe that the boundary condition $\vec{\kappa}=0$ is not well posed. Indeed, the curvature of a curve $f$ can be written as
$$ \vec{\kappa}= \frac{1}{|\pa_x f|^2} \Big( Id_{n \times n} - \pa_s f \otimes \pa_s f \Big) \partial_x^2 f \, .$$
(Here and in the following, for vectors $v,w\in \R^n$ we write $v \otimes w$ to denote the $n\times n$ matrix $v w^{t}$.)
Clearly, the matrix given by the terms between the brackets has $0$ as an eigenvalue. If one instead imposes the boundary condition $\pa_x^2 f=0$ this in particular implies that the curvature is zero and it also gives a well posed problem. 
 For this reason the problem we construct a solution to is in fact \eqref{(P)} with boundary conditions
\begin{equation}\label{eq:nonlinearbc}
\left\{\begin{aligned}
f_i(t,1)& =P_i, &\mbox{ for all }t\in [0,T), i\in \{1, \dots,q\}, \\
\pa_x^2 f_i(t,1) & =0=\pa_x^2 f_i(t,0)  &\mbox{ for all }t\in [0,T), \ i\in \{1, \dots,q\},\\
f_i(t,0) & =f_j(t,0) &\mbox{ for all }t\in [0,T), \ i, j\in \{1, \dots,q\}, \\
\mbox{and }& \sum_{i=1}^q (\nabla_s \vec{\kappa}_{i}(t,0) - \lambda_i \partial_s f_{i}(t,0)) = 0  &\mbox{ for all }t\in [0,T),
\end{aligned} \right.
\end{equation}
instead of \eqref{eq:nonlinearbcgeo}.

In order to find solutions that are $C^{\frac{4+\alpha}{4},4+\alpha}([0,T]\times [0,1])$, $\alpha \in (0,1)$, for some $T>0$ the initial datum has to satisfy some compatibility conditions. Following the notation of \cite[page 98]{Sol} (see also \cite[page 217]{EZ}) we need to impose compatibility conditions of order zero.

\begin{compa}\label{compaorder1}
 We assume that $f_{0,i}:[0,1] \to \R^n$, $f_{0,i} \in C^{4,\alpha}([0,1])$, $ i\in \{1, \dots,q\}$, regular curves, satisfy compatibility condition of order zero for the problem \eqref{(P)}, \eqref{eq:nonlinearbc}. That is, $f_{0,i}$ satisfy the boundary conditions
\begin{equation}\label{eq:nonlinearbcid}
\left\{\begin{aligned}
f_{0,i}(1)& =P_i,  \quad 
\pa_x^2 f_{0,i}(1)=0=\pa_x^2 f_{0,i}(0)  , \  i\in \{1, \dots,q\},\\
f_{0,i}(0) & =f_{0,j}(0), \  i\in \{1, \dots,q\}, \\
\mbox{and }& \sum_{i=1}^q (\nabla_s \vec{\kappa}_{0,i}(0) - \lambda_i \partial_s f_{0,i}(0)) = 0  ,
\end{aligned} \right.
\end{equation}
and at the fixed boundary point $x=1$ the curves satisfy
\begin{equation}\label{eq:compaone} \frac{1}{|\pa_x f_{0,i}|^4} \pa_x^4 f_{0,i}\Big|_{x=1} 
=0\, ,
\end{equation}
(i.e. $\partial_t f_i=0$ at $t=0$, $x=1$) while at the junction point $x=0$
\begin{equation}\label{eq:compatwo} \frac{1}{|\pa_x f_{0,i}|^4}  \pa_x^4 f_{0,i}\Big|_{x=0} 
=\frac{1}{|\pa_x f_{0,j}|^4}  \pa_x^4 f_{0,j}\Big|_{x=0} 
\, ,
\end{equation}
for $i\ne j \in \{1, \dots,q\}$ (i.e. $\partial_t f_i=\pa_t f_j$ at $t=0$, $x=0$).
\end{compa}
\begin{rem}
In the formulas above we have used that $h(f_i)$ is zero at the boundary due to the boundary condition $\pa_x^{2} f_i=0$.
\end{rem}

\begin{teo}\label{teo:STEpde}
Let $n \geq 2$, $q\geq 3$, 
$\alpha\in(0,1)$ and $P_i$, $i\in \{1, \dots,q\}$, be  points in $\R^n$.  
Given $f_{0,i}:[0,1] \to \R^n$, $f_{0,i} \in C^{4,\alpha}([0,1])$, $i\in \{1, \dots,q\}$, regular curves satisfying the Compatibility conditions \ref{compaorder1} and the non-collinearity condition (NC), then there exist $T>0$ and 
regular curves $f_{i} \in C^{\frac{4+\alpha}{4},4+\alpha}([0,T]\times [0,1];\R^n)$, $i\in \{1, \dots,q\}$, 
 such that $f=(f_1,\dots ,f_q)$ is the unique solution of \eqref{(P)} together with the boundary conditions \eqref{eq:nonlinearbc} and the initial condition
\begin{equation}\label{eq:ic}
f_i (t=0)=f_{0,i} \ .
\end{equation}
Moreover, we have instant parabolic smoothing, that is  $f_i \in C^{\infty}((0,T]\times [0,1])$ for any $i\in \{1, \dots,q\}$ and the non-collinearity condition (NC) holds at the triple junction for any time $t \in [0,T]$.
\end{teo}

\section{Proof of Theorem \ref{teo:STEpde}}
\label{sec:PT}


We start by fixing some notation. Given $q$ time dependent curves: $f_{i}: [0,T) \times [0,1] \to \R^{n}$, $i\in \{1, \dots,q\}$, we denote their components by
\begin{equation}\label{eq:notbron}
f_i^{j} \mbox{ with }j=1, \dots, n \mbox{ for each curve } f_i , \quad i\in \{1, \dots,q\} .
\end{equation}
In the following, when it does not create confusion, we will not write the dependence in $t$ and $x$ to keep the notation as slender as possible. As we will see the arguments are independent of $q$ the number of curves. When reading the arguments for the first time it might be useful to consider simply the case $q=3$.

Here $f_{0,i}: [0,1] \to \R^{n}$, $i\in \{1, \dots,q\}$, denote the initial data as given in Theorem~\ref{teo:STEpde} and their components are denoted by $f_{0,i}^j$ according to \eqref{eq:notbron}. Let us recall that these are regular curves and satisfy the Compatibility Conditions \ref{compaorder1} as well as the non-collinearity condition (NC). Set $\delta>0$ as
\begin{equation}\label{eq:defdelta}
\delta:=\min\{|\pa_x f_{0,i}(x)|: x\in [0,1] \mbox{ and }i\in \{1, \dots,q\}\} \, .
\end{equation}

Set for $\alpha \in (0,1)$, and  for some $0<T<1$ and $M>0$  both to be 
chosen later
\begin{equation}\label{eq:X_i}
X_i=\Big\{u\in C^{\frac{4+\alpha}{4}, 4+\alpha} \big([0,T]\times [0,1]; \R^n\big): \; \|u\|_{C^{\frac{4+\alpha}{4},4+\alpha}}\leq M, \, u(0,x)=f_{0,i}\Big\} \, ,
\end{equation} 
for $i\in \{1, \dots,q\}$ (recall Appendix~\ref{sec:fs}, where definition of parabolic H\"older spaces and useful properties are collected).

We proceed now as follows. We first associate a linear system to \eqref{(P)}, \eqref{eq:nonlinearbc} with \eqref{eq:ic} for each $\bar{f} \in \prod_{i=1}^q X_i$ by computing the coefficients at the initial datum and choosing the right hand side depending also on $\bar{f}$ in such a way that a fixed point of the associated solution operator solves the original non-linear problem. Thanks to the non-collinearity condition we show that the linear parabolic system is well-posed and hence we have the solution operator 
\begin{align}\label{eq:descrR}
\mathcal{R}: \prod_{i=1}^q X_i&\to \prod_{i=1}^q  C^{\frac{4+\alpha}{4}, 4+\alpha} \big([0,T]\times [0,1]: \R^n\big)\\ \nonumber
\bar{f}&\mapsto f=\mathcal{R}\bar{f} \, ,
\end{align}
with $\bar{f}=(\bar{f}_1,\dots,\bar{f}_q)$.
Then a fixed point of this map is a solution of \eqref{(P)}, \eqref{eq:nonlinearbc} with \eqref{eq:ic}. 

Below we show in detail how this operator $\mathcal{R}$ is constructed, but first of all let us prove with the following lemma 
that, by choosing first $M$  and then $T$, we can guarantee that the maps $\bar{f}_{i} \in X_{i}$ are regular on the whole considered time interval. The choice of $M$ is specified in \eqref{defM} below.

Here we consider $f_{0,i} \in C^{\frac{4+\alpha}{4},4+\alpha}([0,T]\times [0,1])$ by extending it as a constant function in time.

\begin{lemma}\label{lem:delta}
Let $\bar{f}_i \in X_i$ for some $i\in \{1, \dots,q\}$. Then for any  $T <1$ we have
\begin{equation}\label{eq:stellina} 
\| \pa_x f_{0,i} - \pa_x \bar{f}_i\|_{C^{\frac{\alpha}{4},\alpha}([0,T]\times [0,1])} 
\leq CT^{\frac{\alpha}{4}} \Big(\|\bar{f}_i \|_{C^{\frac{4+\alpha}{4},4+\alpha}([0,T]\times [0,1])} + \|f_{0,i} \|_{C^{4,\alpha}([0,1])}\Big),
\end{equation}
for some universal constant $C$. 
Moreover, there exists $0<T_1 <1$, such that
\begin{equation}\label{eq:immersion}
| \pa_x \bar{f}_i (t,x)|\geq \frac12 \delta >0 \, \mbox{ for all }(t,x)\in [0,T_1]\times [0,1]  ,
\end{equation}
with $\delta$ defined as in \eqref{eq:defdelta}. Here $T_1=T_1(M,\delta,f_0)$ 
and is chosen independently of $i$.
\end{lemma}
\begin{proof}
Since $\bar{f}_i \equiv f_{0,i}$ at $t=0$ the first inequality follows directly from Lemma \ref{lem:propHolder} with $m=0$ and $l=1$ and the fact that we have extended $f_{0,i}$ as a constant in time. 

For \eqref{eq:immersion} we observe that using the first part of the claim and the definition of $\delta$ in \eqref{eq:defdelta} we have
\begin{align*}
| \pa_x \bar{f}_i (t,x)| & \geq | \pa_x f_{0,i} (t,x)| - \| \pa_x f_{0,i} - \pa_x \bar{f}_i \|_{C^0([0,T]\times [0,1])} \\
& \geq \delta - CT^{\frac{\alpha}{4}} \Big(\|\bar{f}_i \|_{C^{\frac{4+\alpha}{4},4+\alpha}} + \|f_{0,i} \|_{C^{4,\alpha}([0,1])}\Big)\\
& \geq \delta - CT^{\frac{\alpha}{4}} \Big(M + \|f_{0,i} \|_{C^{4,\alpha}([0,1])}\Big) \geq \frac12 \delta \, ,
\end{align*} 
by choosing $T_1<1$  such that $T_1^{\frac{\alpha}{4}}C (M + \|f_{0,i} \|_{C^{4,\alpha}([0,1])}) \leq \frac12 \delta$ for any $i\in \{1, \dots ,q\}$.
\end{proof}
Next, let us construct the operator $\mathcal{R}$.
\subsection{The linear system}

\paragraph{The linear PDEs}
Define for $i\in \{1, \dots,q\}$ the coefficients $D_i=D_i(x)$ where
\begin{equation}\label{eq:defDj}
0<D_i := \frac{1}{|\partial_x f_{0,i}|} \, .
\end{equation}
In particular, since the initial curves are sufficiently smooth, there exists $\tilde{\delta}>0$ such that
\begin{equation}\label{eq:deltatilde}
\tilde{\delta}= \min\{D_{i}(x): x\in[0,1] \mbox{ and }i\in \{1, \dots,q\}\} \, .
\end{equation}
Moreover, for $\bar{f}_i \in X_i$ with $i\in \{1, \dots,q\}$ fixed, set $R_i^j=R_i^j(t,x)$,  $j \in \{ 1, \ldots, n\}$, where
\begin{equation}\label{eq:R_i}
R_i^j:= \Big(\frac{1}{|\partial_x f_{0,i}|^4}-\frac{1}{|\partial_x \bar{f}_{i}|^4}\Big)\pa_x^4 \bar{f}_{i}^j \, . 
\end{equation}

Then, for $f_{0,i}$, $i\in \{1, \dots,q\}$, as in Theorem \ref{teo:STEpde} the linear system we consider is
\begin{equation}\label{eq:systemlinear}
\begin{cases}
\pa_t f_i^j+(D_i)^4 \pa_x^4 f_i^j=R_i^j+h^j(\bar f_i) \mbox{ in }(0,T)\times(0,1),\\
f_i^j(t=0,x)=f_{0,i}^j(x), \; x\in [0,1],
\end{cases}
\end{equation}
for $j$ and $i$ as before and with appropriate linear boundary condition that we now derive. Notice that $h^j(\bar f_i)$ denotes the $j-th$ component of the vector $h(\bar f_i)$ defined in \eqref{eq:gi}.

\paragraph{The linear boundary conditions}

We have two boundary points for each curve. One is fixed while the other is the junction point and hence moving. At the fixed boundary points we already have linear boundary conditions (recall \eqref{eq:nonlinearbc}) and hence we can concentrate on the junction point.

At the junction point we have the boundary conditions
$$\left\{\begin{aligned}
\pa_x^2 f_i(t,0) &=0  &\mbox{ for all }t\in (0,T), \ i\in \{1, \dots,q\},\\
f_i(t,0) & =f_j(t,0) &\mbox{ for all }t\in (0,T), \ i,j\in \{1, \dots,q\} \\
\mbox{and } & \sum_{i=1}^q (\nabla_s \vec{\kappa}_{i}(t,0) - \lambda_i \partial_s f_{i}(t,0)) = 0  &\mbox{ for all }t\in (0,T),
\end{aligned} \right.$$
and hence only the last one needs to be linearized. By \eqref{eq:Aderkappa} and since $\pa_x^2 f_i=0$ at $x=0$ we find 
\begin{align}
\nabla_{s}\vec{\kappa}_i & = \pa_s \vec{\kappa}_i=\frac{\pa_x^3 f_{i}}{|\pa_x f_{i}|^{3}} - \langle \frac{\pa_x^3 f_{i}}{|\pa_x f_{i}|^{5}}, \pa_{x} f_i \rangle \pa_{x} f_i  \notag \\
& = \frac{1}{|\pa_x f_{i}|^{3}} \left(Id_{n \times n}- \pa_s f_i \otimes \pa_s f_i\right) \pa_x^3 f_{i} \, . \label{jay}
\end{align}

We consider a linear boundary condition using the initial datum and a given vector $\bar{f} \in \prod_{i=1}^q X_i$ as follows
$$\sum_{i=1}^q E_i \pa_x^3 f_i = b \in \R^n, $$
where the $n \times n$ matrices $E_i$, $i\in \{1, \dots,q\}$, are given by
\begin{equation}\label{eq:matrixE_i}
E_i= (D_i)^3 (Id_{n \times n}- d_i \otimes d_i)\, ,
\end{equation}
where the $d_i$'s are the normalized tangential vectors of the initial data, that is
\begin{equation}\label{eq:vectorsd_i}
d_i:= D_i \pa_x f_{0,i} = \frac{\pa_x f_{0,i}}{|\pa_x f_{0,i}|}. 
\end{equation}
The vector field $b$ is given by
\begin{equation}\label{eq:vectorb}
b = b(\bar{f})= \sum_{i=1}^q (E_i- \bar{E_i}) \pa_x^3 \bar{f}_i + \lambda_i \frac{\pa_x \bar{f}_i}{|\pa_x \bar{f}_i|}, 
\end{equation}
where 
$$\bar{E_i}= \frac{1}{|\pa_x \bar{f}_i|^3} (Id_{n \times n}- \pa_s \bar{f}_i \otimes \pa_s \bar{f}_i) \, .$$
Let us notice that each matrix $E_i$ has determinant zero, but as we will see below, since we consider the sum $\sum_{i} E_{i}$ the boundary condition is still well posed under the assumption of non-collinearity (NC).

Summing up the linear boundary conditions we consider are
\begin{equation}\label{eq:linbc}
\left\{\begin{aligned}
f_i(t,1)& =P_i, &\mbox{ for all }t\in (0,T), \ i\in \{1, \dots,q\}, \\
\pa_x^2 f_i(t,1) & =0=\pa_x^2 f_i(t,0)  &\mbox{ for all }t\in (0,T), \ i\in \{1, \dots,q\},\\
f_i(t,0) & =f_j(t,0) &\mbox{ for all }t\in (0,T), \ i,j\in \{1, \dots,q\},\\
\mbox{and }& \sum_{i=1}^q E_i \pa_x^3 f_i(t,0) = b  &\mbox{ for all }t\in (0,T),
\end{aligned} \right.
\end{equation}
with $b$ defined in \eqref{eq:vectorb}. This choice of $b$ ensures that a fixed point of the associated solution operator will satisfy the boundary conditions \eqref{eq:nonlinearbc}.

\subsection{Existence of solution to the linear problem}
The operator $\mathcal{R}$ is defined as follows: given $\bar{f} \in \prod_{i=1}^q X_i$ we set $\mathcal{R} \bar{f}$ to be the unique solution $f$ of the linear parabolic system  \eqref{eq:systemlinear}, \eqref{eq:linbc}. This can be done according to the next theorem:
\begin{teo}\label{teo:existlin}
Let the assumptions of Theorem \ref{teo:STEpde} hold.  
Let $M>0$ and let $T>0$ be such that the curves belonging to  $X_{i}$ are regular. 
Then for any $\bar{f} \in \prod_{i=1}^q X_i$
there exists $f=(f_1,\dots,f_q)$,   
$f_{i} \in C^{\frac{4+\alpha}{4},4 +\alpha}([0,T]\times [0,1]; \R^n)$, 
$i\in \{1, \dots,q\}$, unique solution of the linear parabolic system \eqref{eq:systemlinear} together with the boundary conditions \eqref{eq:linbc}.

Moreover, there exists a constant $C_0>0$ such that
\begin{align}\label{eq:estSol1}
& \sum_{i=1}^q \| f_i\|_{C^{\frac{4+\alpha}{4},{4+\alpha}}([0,T]\times [0,1])} \\ \nonumber
& \leq C_0 \Big( \sum_{i=1}^q ( \| R_i+h(\bar{f}_i)\|_{C^{\frac{\alpha}{4},\alpha}([0,T]\times [0,1])} + \|f_{0,i}\|_{C^{4,\alpha}([0,1])} + |P_i|)  + \| b\|_{C^{0,\frac{1+\alpha}{4}}([0,T])} \Big) \, .
\end{align}
The constant $C_0$ depends on $n$, $q$, $\delta$ and $\tilde{\delta}$.
\end{teo}
The theorem above give us the solution operator $\mathcal{R}$ described in \eqref{eq:descrR} above.
\subsubsection{Well-posedness of the linear problem}\label{secLP}

Here we check using \cite{Sol} that the linear parabolic problem \eqref{eq:systemlinear} with boundary conditions \eqref{eq:linbc} is well posed. 

First of all, observe that the left hand side of our system \eqref{eq:systemlinear} can be 
written as $\mathcal{L}(x,t,\pa_x,\pa_t) f$ with $f \in \prod_{i=1}^q X_i$ (i.e. $f=(f_1,\dots,f_q)$) and 
\begin{equation}\label{eq:operscriptL}
\mathcal{L}(x,t,\pa_x,\pa_t) 
= \op{diag}(\ell_{kk})_{k=1}^{qn}
\end{equation}
where $$\ell_{kk}(x,t,\pa_x,\pa_t)=\pa_t+(D_i)^4\pa_{x}^4 \mbox{ if }k= (i-1)n+j$$ 
for some $j \in \{1,..,n\}$ and $i \in \{1, \dots ,q\},$
with $D_i$ defined in \eqref{eq:defDj}.
Notice that in \cite[page 8]{Sol} also $\mathcal{L}_0$ the principal part of $\mathcal{L}$ is used. Since here $\mathcal{L}$ coincide with its principal part, for simplicity we work only with $\mathcal{L}$ avoiding $\mathcal{L}_0$ altogether. 

As usual, we associate to these differential operators polynomials with coefficients depending (possibly) on $(t,x)$ by replacing $\pa_x$ by $\mathrm{i}\xi$, $\xi \in \R$ and $\mathrm{i}= \sqrt{-1}$, and $\pa_t$ by $p$, $p \in \mathbb{C}$. Then, 
$$\ell_{kk}(x,t,\mathrm{i}\xi,p)=p+(D_i)^4\xi^4,$$ 
if $k= (i-1)n+j$ for some $j \in \{1,.., n\}$ and $i\in \{1, \dots ,q\},$. In particular, for $\lambda \in \R$
\begin{align*}
\ell_{kk}(x,t,\mathrm{i}\xi\lambda,p\lambda^4)=&p\lambda^4+(D_i)^4(\mathrm{i}\xi\lambda)^4 =\lambda^4 \ell_{kk}(x,t,\mathrm{i}\xi,p).
\end{align*}
In the following, 
\begin{equation}\label{def:L}
L(x,t,\mathrm{i}\xi,p):=\det \mathcal{L}(x,t,\mathrm{i}\xi,p) = \prod_{i=1}^q (p+(D_i)^4\xi^4)^n\, ,
\end{equation}
and $L(x,t,\mathrm{i}\xi \lambda,p\lambda^4)=\lambda^{4\cdot qn}L(x,t,\mathrm{i}\xi,p)$, see \cite[Eq. (1.2)]{Sol}.

Let
\begin{align} \label{eq:Lhat}
\hat{\mathcal{L}} (x,t,\mathrm{i} \xi,p) & \col L(x,t,\mathrm{i} \xi,p) \, \mathcal{L}^{-1}(x,t,\mathrm{i} \xi,p) \\ \nonumber
&= \prod_{i=1}^q (p+(D_i)^4\xi^4)^n  \op{diag}((\ell_{kk})^{-1})_{k=1}^{qn}
= \op{diag}(A_{kk})_{k=1}^{qn}
\end{align}
with 
\begin{equation*}
A_{kk}=A_{kk}(x,t,\mathrm{i} \xi,p)=\frac{\prod_{i=1}^q (p+(D_i)^4\xi^4)^n}{p+(D_l)^4\xi^4} \, ,
\end{equation*}
if $k=(l-1)n+j$ for $l \in \{1, \dots ,q\}$ and $j \in \{1, .., n\}$. 
Since most of the terms are equal, let
$A_1:=A_{11}$, $A_2:= A_{n+1,n+1}$ and so on, i.e. let 
\begin{equation}\label{eq:defAj}
A_i:=A_{(i-1)n+1,(i-1)n+1} \mbox{ for }i\in \{1,\dots,q\} \, .  
\end{equation}
Notice that for $k \in \{1, ..,n\}$ we have 
$A_{(i-1)n+k,(i-1)n+k}=A_{i}$ for $i \in \{i,\dots,q\}$.

\paragraph{Parabolicity condition}
For $\xi \in \R$ and by \eqref{def:L} we see that the roots (in the variable $p$) of the polynomial 
$L(x,t,\mathrm{i}\xi,p)$ 
 are given by
$$
p=-(D_i)^4\xi^4\qquad\forall \,i\in \{1,\dots,q\} \mbox{ and each one with multiplicity }n
$$
and satisfy
$$
\op{Re} p=-(D_i)^4\xi^4\leq-\tilde{\delta}^4\xi^{2b}\text{ with }b=2\quad\forall \, \xi\in\R,\, \forall\,(t,x),$$ 
and  with $\tilde{\delta} = \min\{D_i: i\in \{1,\dots,q\}\}$ (see \eqref{eq:deltatilde}). So the parabolicity condition \cite[Page 8]{Sol} is satisfied and we even have uniform parabolicity.

\paragraph{Initial complementary conditions}

Since our initial conditions are
\begin{align*}
f_i^{j}(t,x)\mid_{t=0}&=f_{0,i}^j(x), \, i\in \{1,\dots,q\}, \, j\in \{1,..,n\},
\end{align*}
the associated matrix is 
\begin{equation}\label{eq:matrixic}
\mathcal{C}_0(x,\pa_x,\pa_t) = Id_{qn \times qn} \, . 
\end{equation}
According to \cite[Page 12]{Sol} we need to verify that the rows of the matrix $\mathcal{D}(x,p)=\mathcal{C}_{0}(x,0,p)\cdot\hat{\mathcal{L}}(x,0,0,p)$ are linearly independent modulo $p^{qn}$ ($r=qn$ in \cite[page 12]{Sol}, since we have $q$ curves). With \eqref{eq:Lhat} and \eqref{eq:defAj} we find that 
$\mathcal{D}(x,p)=
\op{diag}(p^{qn-1})\in\R^{qn\times qn}$ from which the linear independency of the rows immediately follows.

\paragraph{The polynomial $M^{+}$}

Next, consider the polynomial $L=L(x,t,\mathrm{i}\tau,p)$ given in \eqref{def:L} (with $\tau$ instead\footnote{In our case $\xi$ is a vector in $\R$ but in the general setting of \cite{Sol} $\xi \in \R^d$. Now working at the boundary one uses the tangential $\zeta$ and normal $\tau$ directions of the vector $\xi$. Since we are in dimension one there is no tangential direction, i.e. $\zeta=0$.}{ of $\xi$} as in \cite{Sol} to stress that we are now working at the boundary points), that is 
\begin{equation*}
L(x,t,\mathrm{i}\tau,p) = \prod_{i=1}^q (p+(D_i)^4\tau^4)^n\, .
\end{equation*}
As a function of $\tau$, the polynomial $L$ has $2q n$ roots
with positive real part and $2qn$ roots with negative real part provided  $\operatorname{Re}p\geq 0$ and $p \ne 0$ (see \cite[Page~11]{Sol}).  
Indeed, due to the assumptions on $p$ we may write $p=|p| e^{i \theta_p}$ with $-\frac12 \pi \leq \theta_p \leq \frac12 \pi$ and $|p|\ne 0$. Then, the roots have to satisfy for some $i\in \{1,\dots,q\}$
$$ \tau^4 = -\frac{p}{(D_i)^4} \Longrightarrow  \tau^4 = \frac{|p|}{(D_i)^4} e^{\mathrm{i}(\pi + \theta_p)} \, .$$
The (distinct) roots with positive imaginary part are for $i\in \{1,\dots,q\}$
\begin{equation}\label{eq:rootspos}
\tau_{i,1}(x,p) = r_i e^{\mathrm{i} \frac14 (\theta_p+\pi)} \mbox{ and } \tau_{i,2}(x,p) = r_i e^{\mathrm{i} \frac14 (\theta_p+3\pi)}, \mbox{ with }r_i(x,p):= \frac{\sqrt[4]{|p|}}{D_i}, 
\end{equation}
each with multiplicity $n$. With these roots we define the polynomial
\begin{gather*}
 M^+(x,\tau,p)=\prod_{i=1}^q (\tau-\tau_{i,1})^n (\tau-\tau_{i,2})^n\,.
\end{gather*}
For later let us write also the (distinct) roots with negative imaginary part. These are\footnote{In the notation of Solonnikov $L=L(\tau)$ has $b\cdot r=2\cdot qn=2qn$ roots with negative imaginary parts and $2qn$ with negative imaginary parts.}
$$\tau_{i,3}(x,p) = r_i e^{\mathrm{i} \frac14 (\theta_p+5\pi)} \mbox{ and } \tau_{i,4}(x,p) = r_i e^{\mathrm{i} \frac14 (\theta_p+7\pi)}, \; i\in \{1,\dots,q\},
$$
each with multiplicity $n$.

\paragraph{Complementary conditions at the fixed boundary points}
 
By \eqref{eq:linbc}, at $x=1$ the boundary condition system reads: $\mathcal{B} f= \bar{b}$  where $f=(f_{1},\dots,f_q)^{T} \in \R^{qn}$ with
$$\mathcal{B}(x=1,t,\pa_x,\pa_t) = \begin{pmatrix}
B& 0& \dots & 0\\
0&B& \dots & 0\\
0 & 0 & \dots &0\\
0&0& \dots & B\end{pmatrix}$$
a $2qn \times qn$ matrix where $B$ is a $2n \times n$ matrix given by
$$
B=B(x=1,t,\pa_x,\pa_t) = \begin{pmatrix}
Id_{n\times n}\\
Id_{n \times n} \pa_x^2\end{pmatrix} \,$$
so that
$$ B=B(x=1,t,\mathrm{i} \tau,p) = \begin{pmatrix}
Id_{n\times n}\\
-Id_{n \times n} \tau^2\end{pmatrix},$$
and
$$ \bar{b}= (P_1,0,P_2,0, \dots, P_q,0)^{T} \in \R^{2q n} \, . $$
According to \cite[Page~11]{Sol} we need to check that at $x=1$ the rows of the matrix
$$
\mathcal{A}(x=1,t,\mathrm{i}\tau,p):=\mathcal{B}(x=1,t,\mathrm{i}\tau,p)\hat{\mathcal{L}}(x=1,t,\mathrm{i}\tau,p)
$$
are linearly independent modulo $M^+(x=1,\tau,p)$ for 
 $\op{Re}p\geq 0$ and $p\ne0$. In the following we simply write $M^+(\tau)$ since $x$ and $p$ are fixed and there is no explicit dependence on time.

Since $\mathcal{B}$ is a block-matrix and $\hat{\mathcal{L}}$ is a diagonal matrix (see \eqref{eq:Lhat}), $\mathcal{A}$ is also a block matrix. Hence for the linear independence of the rows it is sufficient to consider the different blocks separately, i.e. each curve separately. 
For simplicity we consider the first curve only, that is the first $2n$ rows. We do not need to consider the columns that are identically zero and hence we simply have to consider the rows of the $2n \times n$ matrix
$$ B=B(x=1,t,\mathrm{i} \tau,p) = \begin{pmatrix}
A_1 Id_{n\times n}\\
-A_1 Id_{n \times n} \tau^2\end{pmatrix} \, ,$$
since $A_{kk}=A_1$ for $k\in \{1,..,n\}$ by \eqref{eq:defAj} with
\begin{align*}
A_1(x=1,t, \mathrm{i}\tau,p)&=(p+(D_1)^4\tau^4)^{n-1} \prod_{i=2}^q(p+(D_i)^4\tau^4)^{n}  .
\end{align*}

Now to check the linear independence of the rows modulo $M^+$ we have to 
verify that if there exists $\omega \in \R^{2n}$ such that
$$\omega^{T} \begin{pmatrix}
A_1 Id_{n\times n}\\
-A_1 Id_{n \times n} \tau^2\end{pmatrix}  = (0, \dots, 0) \text{ mod }M^{+}(\tau),$$
then necessarily $\omega=0$.

Now let us recall that $M^{+}$ is the polynomial whose roots are exactly the roots with positive imaginary part of $\prod_{i=1}^q (p + (D_i)^4 \tau^4)^n= A_1 (p+(D_1)^4 \tau^4)$. As a consequence, $A_{1}$ and $M^{+}$ have many factors in common, which we can factor out. More precisely, looking at the first equation of the system above and denoting by $\omega^j$ the $j$-th component of $\omega$, we observe that
\begin{align*}
A_1(\omega^1-\omega^{n+1} \tau^2)=0\text{ mod }M^+(\tau)
\end{align*}
if and only if
\begin{align*}
a_1(\tau)(\omega^1-\omega^{n+1} \tau^2)=0\text{ mod } s_1(\tau)
\end{align*}
where
\begin{align}\label{eq:a_1}
a_1(\tau) & = (\tau-\tau_{1,3}(p))^{n-1} (\tau-\tau_{1,4}(p))^{n-1}  \prod_{i=2}^q (\tau-\tau_{i,3}(p))^n (\tau-\tau_{i,4}(p))^n, \\ \label{eq:s_1}
s_1(\tau)&=(\tau-\tau_{1,1}(p)) (\tau-\tau_{1,2}(p)) \, .
\end{align}
Since $s_1(\tau)$ can not divide $a_1(\tau)$ then it has to divide 
$\omega^1-\omega^{n+1} \tau^2$ that
is also a polynomial of degree two. 
This polynomial has the same zeroes as $s_1(\tau)$ iff
$$\begin{cases}
\omega^1-\omega^{n+1} \mathrm{i} r_1^2 e^{\mathrm{i} \frac{\theta_p}{2}}=0\\
\omega^1+\omega^{n+1} \mathrm{i} r_1^2 e^{\mathrm{i} \frac{\theta_p}{2}}=0
\end{cases}$$
which implies $\omega^1=\omega^{n+1}=0$. Similarly one consider the other components and also the other curves.

\paragraph{Complementary conditions at the junction}

At $x=0$, using \eqref{eq:linbc}, the boundary condition of the linearized system reads: $\mathcal{B} f= \bar{b}$ where $f=(f_{1},f_2,\ldots, f_q)^{T} \in \R^{nq}$ with $\mathcal{B}$ a $(2nq) \times qn$ matrix given by
$$\mathcal{B}(x=0,t,\pa_x,\pa_t) = \begin{pmatrix}
Id_{n \times n}& -Id_{n\times n} & 0 & \ldots & 0\\
Id_{n \times n}& 0 & -Id_{n\times n} & \ldots & 0 \\
\vdots &  \vdots& \vdots& \vdots&  \vdots\\
Id_{n \times n}& 0 & 0&  \ldots &-Id_{n\times n}  \\
Id_{n\times n} \pa_x^2 & 0 &0& \ldots & 0 & \\
0 & Id_{n\times n} \pa_x^2 & 0 &\ldots&  0 \\
\vdots &  \vdots& \vdots& \vdots& \vdots\\
0& 0& 0 &  &Id_{n\times n} \pa_x^2\\
E_1 \pa_x^3&E_2 \pa_x^3 &E_3 \pa_x^3 & \ldots & E_q \pa_x^3
\end{pmatrix}$$
with $E_i$, $i=1,\ldots,q$, $n \times n$ matrices defined in \eqref{eq:matrixE_i} and $\bar{b}= (0,0,\ldots, 0,b)^{T} \in \R^{2nq}$  with $b \in \R^n$ defined in \eqref{eq:vectorb}.  The first $(q-1)n$ rows describe the concurrency condition, the last $n$ rows give the third order boundary condition while the others correspond to the second order boundary condition.

As before, we need to check that at $x=0$ the rows of the matrix
$$
\mathcal{A}(x=0,t,\mathrm{i}\tau,p):=\mathcal{B}(x=0,t,\mathrm{i}\tau,p)\hat{\mathcal{L}}(x=0,t,\mathrm{i}\tau,p)
$$
are linearly independent modulo $M^+(x=0,\tau,p)$ for
 $\op{Re}p\geq 0$ and $p\ne0$. By \eqref{eq:Lhat} we have to study the rows of the matrix
 $\mathcal{A}_0(\tau):=  \mathcal{A}(x=0,t,\mathrm{i}\tau,p)$ with
\begin{align*} 
\mathcal{A}(x=0,t,\mathrm{i}\tau,p) =
\begin{pmatrix}
A_{1} Id_{n \times n}& -A_{2}Id_{n\times n} & 0 & \ldots & 0\\
A_{1}Id_{n \times n}& 0 & -A_{3}Id_{n\times n} & \ldots & 0 \\
\vdots &  \vdots& \vdots& \vdots&  \vdots\\
A_{1}Id_{n \times n}& 0 & 0&  \ldots &-A_{q}Id_{n\times n}  \\
-A_1 \tau^2 Id_{n\times n}  & 0 &0& \ldots & 0 & \\
0 & -A_2 \tau^2 Id_{n\times n}  & 0 &\ldots&  0 \\
\vdots &  \vdots& \vdots& \vdots& \vdots\\
0& 0& 0 &  &-A_q \tau^2 Id_{n\times n} \\
-\mathrm{i}A_1 \tau^3 E_1&-\mathrm{i} A_2\tau^3E_2&-\mathrm{i} A_3\tau^3E_3 & \ldots &-\mathrm{i} A_q\tau^3E_q \end{pmatrix} 
\end{align*}
where the coefficients $A_i$, $i=1,2,\ldots, q$, are defined in \eqref{eq:defAj}.

Let us assume that there exists $\omega \in \R^{2nq}$ such that
$$\omega^{T} \mathcal{A}_0(\tau)= (0,\dots,0) ,\text{ mod }M^{+}(\tau) \, .$$
This is a system of $q n$ equations in $2qn$ variables. 

The factors $A_1 , A_2 , \ldots, A_q$ have many factors in common with the polynomial $M^{+}$ so that we can rewrite the system in the following way. 
The first $n$ equations can be written for $k=1, .., n$ as
\begin{equation}\label{eq:firstnq}
a_1(\tau) \big(\omega^k +\omega^{n+k}+\ldots+ \omega^{n(q-2)+k} - \omega^{(q-1)n+k} \tau^2 - \mathrm{i} \sum_{j=1}^n E_{1,jk} \omega^{(2q-1)n+j} \tau^3\big)=0 ,\text{ mod }s_1(\tau),
\end{equation}
with $a_1,s_1$ defined in \eqref{eq:a_1} and \eqref{eq:s_1} respectively and $E_{1,jk}$ denoting the $j,k$ entry of the matrix $E_1$. Defining $a_2,a_3,\ldots, a_{q}, s_2,s_3, \ldots, s_{q}$ accordingly we find that the $n+k$-th equation, $k=1,..,n$, can be simplified to
\begin{equation}\label{eq:secondnq}
a_2(\tau) \big(-\omega^k - \omega^{qn+k} \tau^2 - \mathrm{i} \sum_{j=1}^n E_{2,jk} \omega^{(2q-1)n+j} \tau^3\big)=0 ,\text{ mod }s_2(\tau).
\end{equation}
Likewise
\begin{align*}
a_3(\tau) &\big(-\omega^{n+k} - \omega^{(q+1)n+k} \tau^2 - \mathrm{i} \sum_{j=1}^n E_{3,jk} \omega^{(2q-1)n+j} \tau^3\big)=0 ,\text{ mod }s_3(\tau),\\
\vdots
\end{align*}
while the $(q-1)n+k$-th  equation for $k=1, ..,n$ can be written as
\begin{equation}\label{eq:thirdnq}
a_q(\tau) \big(-\omega^{n(q-2)+k} - \omega^{2(q-1)n+k} \tau^2 - \mathrm{i} \sum_{j=1}^n E_{q,jk} \omega^{(2q-1)n+j} \tau^3\big)=0 ,\text{ mod }s_q(\tau).
\end{equation}

Since the polynomial $a_i$ does not vanish on the zeros of the polynomial $s_i$ we do not need to consider further the factors $a_i$'s in each one of the equations above. Hence, each algebraic equation above is reduced to the form
\begin{equation*}
 a-bx^2+cx^3 =0 \quad \text{ mod } (x-x_1)(x-x_2) ,
\end{equation*}
with $x_1\ne \pm x_2$. Our idea is now to plug in the zeroes $x_1$ and $x_2$ as done before for the other boundary conditions. We get then the conditions
\begin{equation*}
\begin{cases}
a-bx^2_1+cx^3_1 =0,\\
a-bx^2_2+cx^3_2 =0.
\end{cases}
\end{equation*}
Subtracting the two equations we get first the relation
\begin{equation}
b=\frac{x^2_1+x_1x_2+x^2_2}{x_1+x_2} c
\text{.}
\label{eq:b-c}
\end{equation}
and then, from the first equation,
\begin{equation}
a=\frac{x_1^2 x_2^2}{x_1+x_2} c
\text{.}
\label{eq:a-c} 
\end{equation}

The roots $\tau_{i,1}, \tau_{i,2}$ of $s_i(\tau)$ are given in \eqref{eq:rootspos} and we compute (taking $x_k=\tau_{i,k}$, $k=1,2$)
$$\frac{x^2_1+x_1x_2+x^2_2}{x_1+x_2} = \mathrm{i} r_i\frac{1}{\sqrt{2}} e^{\mathrm{i} \frac14 \theta_p} \quad \mbox{ and }
\quad \frac{x_1^2 x_2^2}{x_1+x_2} = - \mathrm{i} r_i^3\frac{1}{\sqrt{2}} e^{\mathrm{i} \frac34 \theta_p} \, .$$

We now write the $qn$ equations obtained by imposing the algebraic equation \eqref{eq:b-c} to the equations \eqref{eq:firstnq}, \eqref{eq:secondnq}, ..., up to \eqref{eq:thirdnq}. In \eqref{eq:firstnq} we have $b=\omega^{(q-1)n+k}$ and $c=-\mathrm{i}\sum_{j=1}^n E_{1,jk} \omega^{(2q-1)n+j}$, for $k=1,..,n$, so that from \eqref{eq:b-c} we get the $n$ equations
\begin{equation}\label{eq:three1q}
\omega^{(q-1)n+k}= r_1\frac{1}{\sqrt{2}} e^{\mathrm{i} \frac14 \theta_p} \sum_{j=1}^n E_{1,jk} \omega^{(2q-1)n+j} \,, \quad  k=1,..,n .
\end{equation}
Similarly, in \eqref{eq:secondnq}, (and then analogousy up to \eqref{eq:thirdnq}), we have $b=\omega^{qn+k}$ and $c= - \mathrm{i} \sum_{j=1}^n E_{2,jk} \omega^{(2q-1)n+j}$ and \eqref{eq:b-c} implies
\begin{align}\label{eq:three2q} 
\omega^{qn+k}&= r_2\frac{1}{\sqrt{2}} e^{\mathrm{i} \frac14 \theta_p} \sum_{j=1}^n E_{2,jk} \omega^{(2q-1)n+j} \,, \quad  k=1,..,n ,
\\
\vdots & \nonumber \\
\label{eq:three3q}
\omega^{2(q-1)n+k}&= r_q\frac{1}{\sqrt{2}} e^{\mathrm{i} \frac14 \theta_p} \sum_{j=1}^n E_{q,jk} \omega^{(2q-1)n+j} \,, \quad  k=1,..,n .
\end{align}
From these equations we immediately see that the components $\omega^{m}$ for $(q-1)n+1\leq m\leq2(q-1)n$ are determined by the components $\omega^m$ for $ (2q-1)n+1\leq m\leq 2qn$.

The $qn$ equations obtained by imposing the algebraic equation \eqref{eq:a-c} to \eqref{eq:firstnq}, \eqref{eq:secondnq}, ..., up to \eqref{eq:thirdnq} are given by, for $k=1,..,n$
\begin{align}\nonumber
\omega^k + \omega^{n+k} + \ldots + \omega^{n(q-2)+k} & = - r_1^3\frac{1}{\sqrt{2}} e^{\mathrm{i} \frac34 \theta_p} \sum_{j=1}^n E_{1,jk} \omega^{(2q-1)n+j}\\ \label{eq:systemalmostq}
-\omega^k & = - r_2^3\frac{1}{\sqrt{2}} e^{\mathrm{i} \frac34 \theta_p} \sum_{j=1}^n E_{2,jk} \omega^{(2q-1)n+j}\\ \nonumber
-\omega^{k+n} & = - r_3^3\frac{1}{\sqrt{2}} e^{\mathrm{i} \frac34 \theta_p} \sum_{j=1}^n E_{3,jk} \omega^{(2q-1)n+j}\\ \nonumber
\vdots & \nonumber\\
-\omega^{k+ n(q-2)} & = - r_q^3\frac{1}{\sqrt{2}} e^{\mathrm{i} \frac34 \theta_p} \sum_{j=1}^n E_{q,jk} \omega^{(2q-1)n+j} \, . \nonumber
\end{align} 
This system of $qn$ equations implies the system of $n$ equations
$$0=r_1^3\sum_{j=1}^n E_{1,jk} \omega^{(2q-1)n+j}+r_2^3 \sum_{j=1}^n E_{2,jk} \omega^{(2q-1)n+j}+ \ldots +r_q^3\sum_{j=1}^n E_{q,jk} \omega^{(2q-1)n+j}, \quad k=1,..,n, $$
that by definition of $r_i$ in \eqref{eq:rootspos}, of $E_i$ in \eqref{eq:matrixE_i} and since $|p|\ne 0$ can be rewritten as
$$0= (\tilde{E}_1 + \tilde{E}_2 +\ldots +  \tilde{E}_q)v \qquad \mbox{ where} \qquad \tilde{E}_i=Id_{n \times n}- \pi_i,$$
with $\pi_i$ the projection onto the space spanned by $d_i= D_i \pa_x f_{0,i}$ and $v \in \R^n$ with $v^k=\omega^{(2q-1)n+k}$, $k=1,2,...,n$. 
Even better, this can be rewritten as
\begin{equation}\label{eq:superq}
qv - \sum_{i=1}^q \pi_i v =0 \, .
\end{equation}
Since $| \pi_i w|\leq |w|$ for all $w \in\R^n$ and 
$| \pi_i w| = |w|$ iff $w$ and $d_i$ are linearly dependent, we see that the system has a non-trivial solution iff there exists a vector $v \in \R^n$ such that $\pi_i v = v$ for $i=1,2,\ldots, q$. 
This is the case iff $\pi_1=\pi_2=\ldots =\pi_q$, i.e., the vectors $d_i$'s 
are linearly dependent. 

By the non-collinearity condition (NC), the vectors fulfill $\dimension (\Span\{d_1,d_2,\ldots, d_q\}) \geq 2$ 
and hence the only solution to \eqref{eq:superq} is the zero-vector,  that is $\omega^{(2q-1)n+k}=0$ for $k=1, \ldots n$. In turn, $\omega^{m}$ for $(q-1)n+1\leq m \leq 2qn$ are zero (by \eqref{eq:three1q}, \eqref{eq:three2q} to \eqref{eq:three3q}). Then by \eqref{eq:systemalmostq} also $\omega^{m}$ for $1\leq m \leq (q-1)n$ are zero.
The rows are linearly independent and hence the complementary conditions are satisfied also at the junction.

\subsubsection{Proof of Theorem \ref{teo:existlin}}\label{sec:existlin}

By the regularity assumptions of the initial data $f_{0,i}$ and $\bar{f}_{i} \in X_{i}$ and using the properties of the parabolic spaces collected in Appendix~\ref{sec:fs} one can readily check that the regularity assumptions  required by \cite[Thm.4.9, page 121]{Sol}  for the coefficients of the elliptic operator
$\mathcal{L}$ in \eqref{eq:operscriptL}, and those for the coefficients of the boundary operators $\mathcal{B}$ and $\mathcal{C}_{0}$ are satisfied.

Next, thanks to the choice of the space $X_i$ in \eqref{eq:X_i} (precisely, being each function $\bar{f}_i$ equal to $f_{0,i}$ at time $t=0$) and having linearized at the initial datum one sees that since $f_{0,i}$, $i\in \{1,\dots,q\}$, satisfy the Compatibility conditions \ref{compaorder1} (i.e. the compatibility conditions of order zero of the non-linear problem) than $f_{0,i}$, $i\in \{1,\dots,q\}$, satisfies also the compatibility condition of order zero of the linear problem. Precisely, the boundary conditions
\begin{equation*}
\left\{\begin{aligned}
f_{0,i}(1)& =P_i, \;  \pa_x^2 f_{0,i}(1)  =0=\pa_x^2 f_{0,i}(0)  , \ i\in \{1,\dots,q\}, \\
f_{0,i}(0) & =f_{0,j}(0), \ i,j\in \{1,\dots,q\}, \\
\mbox{and }& \sum_{i=1}^q E_i \pa_x^3 f_{0,i} = b  ,
\end{aligned} \right.
\end{equation*}
are satisfied since at $t=0$, $\bar{E_i}=E_i$. Similarly, since $R_i=0$ at time $t=0$ (recall \eqref{eq:R_i}) we see that \eqref{eq:compaone} and \eqref{eq:compatwo} give also the remaining compatibility condition of order zero.

Since linear problem is well posed by the considerations in the previous section, the claim follows by \cite[Thm.4.9, page 121]{Sol} or \cite[Thm. VI.21]{EZ}.

\subsection{Existence by Banach fixed point theorem}

Let us recall that $f_{0,i} \in C^{\frac{4+\alpha}{4},4+\alpha}([0,T]\times [0,1])$ by extending it as a constant function in time.

\begin{rem}\label{rem:algebra}
For $m \in \mathbb{N}$, $a,b \in (0,\infty)$ one has
$$ \frac{1}{a^m}-\frac{1}{b^m}= (b-a) \frac{p_{m-1}(a,b)}{a^m b^m}, $$
with $p_{m-1}$ a polynomial in $a,b$ of degree $m-1$.
\end{rem}

\begin{lemma}\label{lem:hilfsatz}
Let $f_{0,i} \in C^{4,\alpha}([0,1])$ and $\bar{f}_i, \bar{g}_i \in X_i$ and $\delta$ as defined in \eqref{eq:defdelta}. Then, for $T<1$ we have
\begin{equation}\label{eq:hol1}
\| |\pa_x f_{0,i} |- |\pa_x \bar{f}_i|\|_{C^{\frac{\alpha}{4},\alpha}([0,T]\times[0,1])} \leq C T^{\frac{\alpha}{4}} \Big(\|\bar{f}_i \|_{C^{\frac{4+\alpha}{4},4+\alpha}([0,T]\times[0,1])} + \|f_{0,i} \|_{C^{4,\alpha}([0,1])}\Big)^3\,  ,
\end{equation}
with $C= C(n,\delta)$. Moreover, for $m \in \mathbb{N}$ 
and any $T \leq T_1$  (with $T_1$ as defined in Lemma~\ref{lem:delta}) we have
\allowdisplaybreaks{\begin{align*}
\Big\| \frac{1}{|\pa_x f_{0,i} |^m}- \frac{1}{|\pa_x \bar{f}_i|^m} \Big\|_{C^{\frac{\alpha}{4},\alpha}([0,T]\times [0,1])} & \leq C   T^{\frac{\alpha}{4}} 
\\
\Big\| \frac{1}{|\pa_x f_{0,i} (\cdot, x)|^m}- \frac{1}{|\pa_x \bar{f}_i (\cdot, x)|^m}\Big\|_{C^{0,\frac{1+\alpha}{4}}([0,T])} & \leq C T^{\frac{\alpha}{4}} 
\end{align*}
for any $x \in [0,1]$ and with  $C=C(n,m,\delta, \|\bar{f}_i \|_{C^{\frac{4+\alpha}{4},4+\alpha}([0,T]\times [0,1])}, \|f_{0,i} \|_{C^{4,\alpha}([0,1])})$
as well as 
\begin{align*}
\Big\| \frac{1}{|\pa_x \bar{f}_i |^m}- \frac{1}{|\pa_x \bar{g}_i|^m}\Big\|_{C^{\frac{\alpha}{4},\alpha}([0,T]\times [0,1])} & \leq  C  T^{\frac{\alpha}{4}}  \| \bar{f}_i - \bar{g}_i\|_{C^{\frac{4+\alpha}{4},4+\alpha}([0,T]\times [0,1])}\, ,\\
\Big\| \frac{1}{|\pa_x \bar{f}_i (\cdot, x)|^m}- \frac{1}{|\pa_x \bar{g}_i (\cdot, x)|^m}\Big\|_{C^{0,\frac{1+\alpha}{4}}([0,T])} &\leq C  T^{\frac{\alpha}{4}}  \| \bar{f}_i - \bar{g}_i\|_{C^{\frac{4+\alpha}{4},4+\alpha}([0,T]\times [0,1])}\,,
\end{align*}}
again for  $x \in [0,1]$ and  with $C=C(n,m,\delta, \|\bar{f}_i \|_{C^{\frac{4+\alpha}{4},4+\alpha}([0,T]\times [0,1])}, \|\bar{g}_i \|_{C^{\frac{4+\alpha}{4},4+\alpha}([0,T]\times [0,1])})$.
\end{lemma}
\begin{proof}  The first inequality is a direct consequence of Lemma \ref{lem:HolderChaos}, Remark \ref{rem:Holder}, \eqref{eq:stellina} and the definition of $\delta$ in \eqref{eq:defdelta}. 
Indeed, 
\begin{align*}
\left\| |\pa_x f_{0,i} |- |\pa_x \bar{f}_i| \right \|_{C^{\frac{\alpha}{4},\alpha}} &\leq C(n)  \left\| \frac{1}{|\pa_x f_{0,i}|} \right \|_{C^{0}([0,1])}^2 \Big(\|\bar{f}_i \|_{C^{\frac{4+\alpha}{4},4+\alpha}} + \|f_{0,i} \|_{C^{4,\alpha}([0,1])}\Big)^2 \\
& \quad \times \| \pa_x f_{0,i} - \pa_x \bar{f}_i\|_{C^{\frac{\alpha}{4},\alpha}}\\
& \leq  C(n) T^{\frac{\alpha}{4}} \left\| \frac{1}{|\pa_x f_{0,i}|} \right \|_{C^{0}([0,1])}^2 \Big(\|\bar{f}_i \|_{C^{\frac{4+\alpha}{4},4+\alpha}} + \|f_{0,i} \|_{C^{4,\alpha}([0,1])}\Big)^3 \, .
\end{align*}
Next, let us denote with $p_k(\cdots)$ a  polynomial of degree at most $k$ in its given variables. For the second and fourth inequalities using Remark \ref{rem:algebra}, Lemmas \ref{lem:delta}, \ref{lem:Holder}, \ref{lem:HolderChaos} and \eqref{eq:hol1} we find for $T \leq T_1$
\begin{align*}
& \Big\| \frac{1}{|\pa_x f_{0,i} |^m}- \frac{1}{|\pa_x \bar{f}_i|^m} \Big\|_{C^{\frac{\alpha}{4},\alpha}([0,T]\times [0,1])}
\\
& \qquad \leq  \| |\pa_x f_{0,i} |- |\pa_x \bar{f}_i|\|_{C^{\frac{\alpha}{4},\alpha}}
\Big\| \frac{p_{m-1}(|\pa_x f_{0,i} |,|\pa_x \bar{f}_i|)}{|\pa_x f_{0,i} |^m |\pa_x \bar{f}_i|^m} \Big\|_{C^{\frac{\alpha}{4},\alpha}} \\
& \qquad \leq  C T^{\frac{\alpha}{4}} \Big(\|\bar{f}_i \|_{C^{\frac{4+\alpha}{4},4+\alpha}} + \|f_{0,i} \|_{C^{4,\alpha}([0,1])}\Big)^3 
p_{3m-1}(\|\pa_x f_{0,i} \|_{C^{0,\alpha}([0,1])},\|\pa_x \bar{f}_i\|_{C^{\frac{\alpha}{4},\alpha}})\\
& \qquad \leq  C T^{\frac{\alpha}{4}} p_{3m+2}(\|f_{0,i} \|_{C^{4,\alpha}([0,1])},\|\bar{f}_i\|_{C^{\frac{4+\alpha}{4},4+\alpha}}) \, ,
\end{align*}
with $C=C(n,m, \delta)$. 
With the same ideas
\begin{align*}
& \Big\| \frac{1}{|\pa_x \bar{f}_i |^m}- \frac{1}{|\pa_x \bar{g}_i|^m} \Big\|_{C^{\frac{\alpha}{4},\alpha}([0,T]\times [0,1])}
\\
& \qquad \leq C  p_{3m-1}(\| \bar{f}_i \|_{C^{\frac{4+\alpha}{4},4+\alpha}} , \|\bar{g}_i\|_{C^{\frac{4+\alpha}{4},4+\alpha}} ) \| |\pa_x \bar{f}_i| - |\pa_x \bar{g}_i| \|_{C^{\frac{\alpha}{4},\alpha}} \\
& \qquad \leq C p_{3m+1}(\| \bar{f}_i \|_{C^{\frac{4+\alpha}{4},4+\alpha}} , \|\bar{g}_i\|_{C^{\frac{4+\alpha}{4},4+\alpha}} )  \| \pa_x \bar{f}_i - \pa_x \bar{g}_i \|_{C^{\frac{\alpha}{4},\alpha}} \\
& \qquad \leq C T^{\frac{\alpha}{4}} p_{3m+1}(\| \bar{f}_i \|_{C^{\frac{4+\alpha}{4},4+\alpha}} , \|\bar{g}_i\|_{C^{\frac{4+\alpha}{4},4+\alpha}} )  \| \bar{f}_i - \bar{g}_i \|_{C^{\frac{4+\alpha}{4},4+\alpha}}  \, .
\end{align*}
The statement for $x$ fixed are obtained similarly using also \eqref{eq:estHolbdy}. For instance, with Remarks \ref{rem:algebra} and \ref{rem:Holder}, Lemmas \ref{lem:Holder}, \ref{lem:HolderChaos} , \ref{lem:delta}
, \ref{lem:propHolder} and \eqref{eq:estHolbdy} we obtain
 \allowdisplaybreaks{\begin{align*}
 \Big\|& \frac{1}{|\pa_x f_{0,i} (\cdot, x)|^m}- \frac{1}{|\pa_x \bar{f}_i (\cdot, x)|^m}\Big\|_{C^{0,\frac{1+\alpha}{4}}([0,T])}  \\
& \leq 
\Big\| |\pa_x f_{0,i} (\cdot, x)|- |\pa_x \bar{f}_i (\cdot, x)|\Big\|_{C^{0,\frac{1+\alpha}{4}}([0,T])}
\Big\| \frac{p_{m-1}(|\pa_x f_{0,i} (\cdot, x)|,|\pa_x \bar{f}_i (\cdot, x)|)}{|\pa_x f_{0,i} (\cdot, x)|^m|\pa_x \bar{f}_i (\cdot, x)|^m}\Big\|_{C^{0,\frac{1+\alpha}{4}}([0,T])} \\
& \leq C \Big\| \frac{1}{|\pa_x f_{0,i} (\cdot, x)|}\Big\|^2_{C^{0}([0,T])} (\| \pa_x f_{0,i} (\cdot, x)\|_{C^{0,\frac{1+\alpha}{4}}([0,T])} + \| \pa_x \bar{f}_{i} (\cdot, x)\|_{C^{0,\frac{1+\alpha}{4}}([0,T])})^2 \\
& \qquad \times \| \pa_x f_{0,i} (\cdot, x)- \pa_x \bar{f}_i (\cdot, x) \|_{C^{0,\frac{1+\alpha}{4}}([0,T])}\\
& \qquad \times  p_{3m-1}(\| \pa_x f_{0,i} (\cdot, x)\|_{C^{0,\frac{1+\alpha}{4}}([0,T])},\| \pa_x \bar{f}_{i} (\cdot, x)\|_{C^{0,\frac{1+\alpha}{4}}([0,T])})\\
& \leq C p_{3m+1}(\| \pa_x f_{0,i} (\cdot, x)\|_{C^{0,\frac{1+\alpha}{4}}([0,T])},\| \pa_x \bar{f}_{i} (\cdot, x)\|_{C^{0,\frac{1+\alpha}{4}}([0,T])}) \\
& \qquad \times  \| \pa_x f_{0,i} (\cdot, x)- \pa_x \bar{f}_i (\cdot, x) \|_{C^{\frac{1+\alpha}{4},1+\alpha}([0,T]\times[0,1])} \\
& \leq C T^{\frac{\alpha}{4}} p_{3m+2}(\|f_{0,i}\|_{C^{4,\alpha}([0,1])},\|\bar{f}_{i}\|_{C^{\frac{4+\alpha}{4}, 4 +\alpha}([0,T]\times[0,1])}) \, ,
\end{align*}}
where once again the constant depends on $n,m$ and $\delta$.
\end{proof}

\paragraph{Strict contraction}

Consider the solution operator $\mathcal{R}$ given in \eqref{eq:descrR}.  
Our aim is to show that $\mathcal{R}$ is a (strict) contraction, that is with $f= \mathcal{R}(\bar{f})$, $g= \mathcal{R}(\bar{g})$ the contraction estimate, 
\begin{align}\label{eq:Banach}
\| f-g \|_{C^{\frac{4+\alpha}{4},4+\alpha}([0,T]\times[0,1])} 
\le C_1 T^{\frac{\alpha}{4}} \|\bar f-\bar g \|_{C^{\frac{4+\alpha}{4},4+\alpha}([0,T]\times[0,1])} 
\end{align} 
holds for some constant $C_1=C_1(n,q,\delta,\tilde{\delta},\|f_0\|_{C^{4,\alpha}([0,1])},M)$ and $T \leq T_1$ with $T_1$ defined in Lemma~\ref{lem:delta}. Here $M$ is the constant in \eqref{eq:X_i}.

Observe that $f-g$ fulfills the parabolic linear system, 
\begin{equation}\label{eq:lin_(f-g)}
\pa_t(f_i-g_i)+D_i^4 \pa_x^4 (f_i-g_i)
=R_i(\bar f_i)+h(\bar f_i)-R_i(\bar g_i)-h(\bar g_i),
\end{equation}
for all $(t,x)\in(0,T)\times(0,1)$, $i\in \{1,\dots,q\}$, with $R_i$, $h$ defined in \eqref{eq:R_i} and \eqref{eq:gi} respectively, together with  the initial condition $(f_i-g_i)(t=0)=0$ and the boundary conditions
\begin{equation*}
\left\{\begin{aligned}
(f_i-g_i)(t,1)& =0, &\forall \, t\in[0,T], i\in \{1,\dots,q\}, \\ 
(f_i-g_i)(t,0)& =(f_j-g_j)(t,0), &\forall \, t\in[0,T], i,j\in \{1,\dots,q\}, \\
\pa_x^2 (f_i - g_i)(t,0) & =0 = \pa_x^2 (f_i - g_i)(t,1),  &\forall \, t\in[0,T], i\in \{1,\dots,q\}, \\
\sum_{i=1}^q E_i \pa_x^3 (f_i- g_i)(t,0) & = b(\bar f)-b(\bar g),  &\forall \, t\in[0,T], 
\end{aligned} \right.
\end{equation*}
where the boundary term $b$ is defined in (\ref{eq:vectorb}).  

By the same arguments as in Section \ref{sec:existlin} the linear problem is well posed and the regularity assumptions on the coefficients are satisfied. Moreover, since $\bar{f}=\bar{g}$ at $t=0$ we see that the zero initial datum satisfies the compatibility conditions of order zero and hence $f-g$ is the solution given by \cite[Thm.4.9, page 121]{Sol} or \cite[Thm. VI.21]{EZ} of \eqref{eq:lin_(f-g)} with the boundary conditions given above. Moreover, the same theorems give the following estimate 
\begin{align}\label{eq:Solo_(f-g)}
&\sum_{i=1}^q \| f_i-g_i\|_{C^{\frac{\alpha+4}{4}, \alpha+4}([0,T]\times[0,1])} 
\\ \nonumber
&\leq C_0 \Big( \sum_{i=1}^q (
 \| R_i(\bar f_i)+h(\bar f_i)-R_i(\bar g_i)-h(\bar g_i)\|_{C^{\frac{\alpha}{4}, \alpha}([0,T]\times[0,1])}) + \| b(\bar f)-b(\bar g) \|_{C^{0,\frac{\alpha+1}{4}}([0,T])}\Big) \, ,
\end{align}
where $C_0=C_0(n,q,\delta,\tilde{\delta})$ is the constant in Theorem \ref{teo:existlin}. To obtain inequality \eqref{eq:Banach}, we need to estimate the terms on the right hand side of \eqref{eq:Solo_(f-g)}. 
First of all, for any $i\in \{1,\dots,q\}$, by applying triangle inequality of 
H\"{o}lder norms, Remark \ref{rem:Holder} and Lemmas \ref{lem:Holder}, \ref{lem:hilfsatz} 
we have for $T\leq T_1$ with $T_1$ from Lemma \ref{lem:delta}
\begin{align}\nonumber
&\| R_i(\bar f_i)-R_i(\bar g_i)\|_{C^{\frac{\alpha}{4}, \alpha}([0,T]\times[0,1])} 
\\ \nonumber
& \leq  C
\left\| \frac{1}{|\pa_x f_{0,i}|^4}-\frac{1}{|\pa_x \bar f_{i}|^4} \right\|_{C^{\frac{\alpha}{4}, \alpha}([0,T]\times[0,1])} 
\| \pa_x^4 \bar f_i - \pa_x^4 \bar g_i \|_{C^{\frac{\alpha}{4}, \alpha}([0,T]\times[0,1])} 
\\  \nonumber
&\quad~~ + C\left\| \frac{1}{|\pa_x \bar g_i|^4}-\frac{1}{|\pa_x \bar f_{i}|^4} \right\|_{C^{\frac{\alpha}{4}, \alpha}([0,T]\times[0,1])} 
\| \pa_x^4 \bar g_i \|_{C^{\frac{\alpha}{4}, \alpha}([0,T]\times[0,1])} \\ \label{eq:R(f)-R(g)}
\quad &\leq C T^\frac{\alpha}{4} \| \bar f_i-\bar g_i \|_{C^{\frac{\alpha+4}{4}, \alpha+4}([0,T]\times[0,1])}\, ,
\end{align} 
where $C=C(n,\delta, \|f_{0,i} \|_{C^{4,\alpha}([0,1])},\|\bar{f}_i\|_{C^{\frac{4+\alpha}{4},4+\alpha}},\|\bar{g}_i\|_{C^{\frac{4+\alpha}{4},4+\alpha}})$. 

With the same ideas, since $\bar{f}_i=\bar{g}_i$ at $t=0$, using Lemmas \ref{lem:Holder},  \ref{lem:propHolder}, \ref{lem:hilfsatz}, \ref{lem:delta} we have for $T\leq T_1$
\begin{equation}\label{eq:h(f)-h(g)}
\begin{aligned}
&\| h(\bar f_i)-h(\bar g_i)\|_{C^{\frac{\alpha}{4}, \alpha}([0,T]\times[0,1])} 
\\
&\leq C \left( \sum_{k=1}^3 \| \pa_x^k \bar f_i- \pa_x^k \bar g_i\|_{C^{\frac{\alpha}{4}, \alpha}([0,T]\times[0,1])}
+ \sum_{k=1}^4 \left \| \frac{1}{|\pa_x \bar{g}_{i}|^{2k}}-\frac{1}{|\pa_x \bar f_{i}|^{2k}} \right \|_{C^{\frac{\alpha}{4}, \alpha}([0,T]\times[0,1])} 
\right)\\
& \leq C T^{\frac{\alpha}{4}} \| \bar f_i-\bar g_i \|_{C^{\frac{\alpha+4}{4}, \alpha+4}([0,T]\times[0,1])}
\end{aligned} 
\end{equation}
where $C=C(n,\delta, \lambda_{i}, \|\bar{f}_i\|_{C^{\frac{4+\alpha}{4},4+\alpha}},\|\bar{g}_i\|_{C^{\frac{4+\alpha}{4},4+\alpha}})$. 

For the estimates of the boundary terms, 
$\| b(\bar f)-b(\bar g) \|_{C^{0,\frac{\alpha+1}{4}}([0,T])}$, 
we start from the term multiplied by $\lambda_{i}$, see \eqref{eq:vectorb}. 
With help from Lemmas \ref{lem:Holder}, \ref{lem:delta}, \ref{lem:HolderChaos}, \ref{lem:hilfsatz} and Remark \ref{rem:Holder}, we have  
\begin{equation}\label{eq:(b(f)-b(g))-1}
\begin{aligned}
& \Big\| \frac{\pa_x \bar f_i}{|\pa_x \bar f_i|}- \frac{\pa_x \bar g_i}{|\pa_x \bar g_i|} \Big\|_{C^{0,\frac{\alpha+1}{4}}([0,T])}
\\
&\leq C
\Big\| \frac{1}{|\pa_x \bar f_i|}\Big\|_{C^{0,\frac{\alpha+1}{4}}([0,T])} 
\|\pa_x\bar f_i-\pa_x \bar g_i \|_{C^{0,\frac{\alpha+1}{4}}([0,T])}
\\
& \qquad + C\Big\| \frac{1}{|\pa_x \bar f_i|}- \frac{1}{|\pa_x \bar g_i|}\Big\|_{C^{0,\frac{\alpha+1}{4}}([0,T])}
\|\pa_x\bar g_i \|_{C^{0,\frac{\alpha+1}{4}}([0,T])} \\ 
& \leq C T^{\frac{\alpha}{4}} \| \bar f_i-\bar g_i \|_{C^{\frac{\alpha+4}{4}, \alpha+4}([0,T]\times[0,1])} 
\, ,
\end{aligned} 
\end{equation}
where $C=C(n,\delta, \|\bar{f}_i\|_{C^{\frac{4+\alpha}{4},4+\alpha}},\|\bar{g}_i\|_{C^{\frac{4+\alpha}{4},4+\alpha}})$.

For the highest-order terms at the boundary we compute 
\begin{equation}\label{eq:(b(f)-b(g))-2}
\begin{aligned}
& \| \sum_{i=1}^q (E_i- \bar{E_i}(\bar f_i)) \pa_x^3 \bar{f}_i
- (E_i- \bar{E_i}(\bar g_i)) \pa_x^3 \bar{g}_i\|_{C^{0,\frac{\alpha+1}{4}}([0,T])}
\\
&\leq 
\sum_{i=1}^q 
\| (E_i- \bar{E_i}(\bar f_i)) (\pa_x^3 \bar{f}_i - \pa_x^3 \bar{g}_i)\|_{C^{0,\frac{\alpha+1}{4}}([0,T])}
+ \sum_{i=1}^q 
\| (\bar{E_i}(\bar f_i)- \bar{E_i}(\bar g_i)) \pa_x^3 \bar{g}_i\|_{C^{0,\frac{\alpha+1}{4}}([0,T])}\,.
\end{aligned} 
\end{equation}
Note that we may split the matrix term as 
\begin{equation}\label{eq:E_0-E_1}
\begin{aligned}
&E_i- \bar{E_i}(\bar f_i) 
\\
&= \left(\frac{1}{|\pa_x f_{0,i}|^3} - \frac{1}{|\pa_x \bar{f}_i|^3}\right) Id_{n \times n}
-\left( \frac{1}{|\pa_x f_{0,i}|^3} d_i \otimes d_i
-\frac{1}{|\pa_x \bar f_i|^3}\pa_s \bar{f}_i \otimes \pa_s \bar{f}_i \right)
\\ 
&= \left(\frac{1}{|\pa_x f_{0,i}|^3} - \frac{1}{|\pa_x \bar{f}_i|^3}\right) Id_{n \times n}
-\left( \frac{1}{|\pa_x f_{0,i}|^5}-\frac{1}{|\pa_x \bar f_i|^5}\right) \pa_x f_{0,i} \otimes \pa_x f_{0,i} \\ 
&~~~~~~+ \frac{1}{|\pa_x \bar f_i|^5} 
\pa_x f_{0,i} \otimes (\pa_x f_{0,i}  - \pa_x \bar f_i )
+ \frac{1}{|\pa_x \bar f_i|^5} 
( \pa_x f_{0,i} -\pa_x \bar f_i ) \otimes \pa_x \bar f_i \,.
\end{aligned} 
\end{equation} 
By applying the linear algebra, $(\vec u\otimes \vec v)\vec w=\vec u \langle \vec v, \vec w\rangle$, 
from \eqref{eq:E_0-E_1} we have using Lemmas~\ref{lem:Holder}, \ref{lem:delta}, \ref{lem:hilfsatz}, \ref{lem:propHolder} and Remark \ref{rem:Holder} (for simplicity  here we mostly write $C^{0,\frac{\alpha+1}{4}}$ instead of $C^{0,\frac{\alpha+1}{4}}([0,T])$)
\allowdisplaybreaks{
\begin{align*}
& \| (E_i- \bar{E_i}(\bar f_i)) \cdot (\pa_x^3 \bar{f}_i - \pa_x^3 \bar{g}_i)\|_{C^{0,\frac{\alpha+1}{4}}([0,T])}
\\
&\leq C \left\| \frac{1}{|\pa_x f_{0,i}|^3} - \frac{1}{|\pa_x \bar{f}_i|^3} \right \|_{C^{0,\frac{\alpha+1}{4}}} 
\| \pa_x^3 \bar{f}_i - \pa_x^3 \bar{g}_i \|_{C^{0,\frac{\alpha+1}{4}}} 
\\ 
&+ C \left \| \frac{1}{|\pa_x f_{0,i}|^5} - \frac{1}{|\pa_x \bar{f}_i|^5} \right \|_{C^{0,\frac{\alpha+1}{4}}} 
\| \pa_x f_{0,i}\|^2_{C^{0,\frac{\alpha+1}{4}}} 
\| \pa_x^3 \bar{f}_i - \pa_x^3 \bar{g}_i \|_{C^{0,\frac{\alpha+1}{4}}} 
\\ 
&+ C \left \| \frac{1}{|\pa_x \bar{f}_i|^5}  \right \|_{C^{0,\frac{\alpha+1}{4}}} 
\| \pa_x f_{0,i}\|_{C^{0,\frac{\alpha+1}{4}}} 
\| \pa_x f_{0,i} - \pa_x \bar{f}_i \|_{C^{0,\frac{\alpha+1}{4}}} 
\| \pa_x^3 \bar{f}_i - \pa_x^3 \bar{g}_i \|_{C^{0,\frac{\alpha+1}{4}}} 
\\ 
&+ C \left \| \frac{1}{|\pa_x \bar{f}_i|^5} \right  \|_{C^{0,\frac{\alpha+1}{4}}} 
\| \pa_x \bar f_i\|_{C^{0,\frac{\alpha+1}{4}}} 
\| \pa_x f_{0,i} - \pa_x \bar{f}_i \|_{C^{0,\frac{\alpha+1}{4}}} 
\| \pa_x^3 \bar{f}_i - \pa_x^3 \bar{g}_i \|_{C^{0,\frac{\alpha+1}{4}}} 
\\
& \leq C T^{\frac{\alpha}{4}}
\| \bar f_i-\bar g_i \|_{C^{\frac{4+\alpha}{4}, 4+\alpha}([0,T]\times[0,1])} 
\, ,
\end{align*}}  
with $C=C(n, \delta, \|\bar f\|_{C^{\frac{4+\alpha}{4}, 4+\alpha}}, \| f_{0}\|_{C^{4, \alpha} ([0,1])}, 
\|\bar g\|_{C^{\frac{4+\alpha}{4}, 4+\alpha}})$.

Similarly, we may apply the same trick of estimates to the second term in \eqref{eq:(b(f)-b(g))-2}. 
More precisely, writing
\begin{align*}
(\bar{E_i}(\bar{f}_i)- \bar{E_i}(\bar{g}_i)) \pa_x^3 \bar{g}_i & = \Big(\frac{1}{|\pa_x \bar{f}_i|^3}-\frac{1}{|\pa_x \bar{g}_i|^3} \Big)  \pa_x^3 \bar{g}_i + \Big(\frac{1}{|\pa_x \bar{f}_i|^5}-\frac{1}{|\pa_x \bar{g}_i|^5} \Big) \pa_x \bar{f}_i \otimes \pa_x \bar{f}_i  \pa_x^3 \bar{g}_i \\
& \quad + \frac{1}{|\pa_x \bar{g}_i|^5}  \left( (\pa_x \bar{f}_i - \pa_x \bar{g}_i )\otimes \pa_x \bar{f}_i  +   \pa_x \bar{g}_i \otimes ( \pa_x \bar{f}_i - \pa_x \bar{g}_i) \right)  \pa_x^3 \bar{g}_i
\end{align*}
using Remark \ref{rem:Holder}, Lemmas \ref{lem:delta}, \ref{lem:hilfsatz}, \ref{lem:Holder} and \eqref{eq:estHolbdy} we obtain 
\begin{equation*}
\|(\bar{E_i}(\bar f_i)- \bar{E_i}(\bar g_i)) \pa_x^3 \bar{g}_i \|_{C^{0,\frac{\alpha+1}{4}}([0,T])}
\leq 
C T^{\frac{\alpha}{4}} \| \bar f_i-\bar g_i \|_{C^{\frac{4+\alpha}{4}, 4+\alpha}([0,T]\times[0,1])} 
\,,
\end{equation*}
with $C= C(n,\delta, \|\bar{f}_i\|_{C^{\frac{4+\alpha}{4}, 4+\alpha}},\|\bar{g}_i\|_{C^{\frac{4+\alpha}{4}, 4+\alpha}})$. 
Combining the previous estimates we therefore infer 
\begin{align}\label{estbcont}
\| b(\bar f)-b(\bar g) \|_{C^{0,\frac{\alpha+1}{4}}([0,T])} \leq C T^{\frac{\alpha}{4}} 
\sum_{i=1}^{q}\| \bar f_i-\bar g_i \|_{C^{\frac{4+\alpha}{4}, 4+\alpha}([0,T]\times[0,1])}
\end{align}
with $C=C(\lambda,n, \delta, \|\bar f\|_{C^{\frac{4+\alpha}{4}, 4+\alpha}([0,T]\times[0,1])}, \| f_{0}\|_{C^{4, \alpha} ([0,1])}, 
\|\bar g\|_{C^{\frac{4+\alpha}{4}, 4+\alpha}([0,T]\times[0,1])})$.

From \eqref{eq:Solo_(f-g)} together with \eqref{estbcont}, \eqref{eq:h(f)-h(g)} and \eqref{eq:R(f)-R(g)}, we obtain \eqref{eq:Banach} since $\|\bar{f}_i\|_{C^{\frac{4+\alpha}{4}, 4+\alpha}}$ and $\|\bar{g_i}\|_{C^{\frac{4+\alpha}{4}, 4+\alpha}}$ are bounded by $M$ by definition of $X_i$.

\paragraph{Self-map} 
We show now that $\mathcal{R}$ defined in (\ref{eq:descrR}) indeed maps $\prod_{i=1}^q X_i$ into itself by choosing first $M$ and then $T$ sufficiently small. Given $\bar{f} \in \prod_{i=1}^q X_i$ by Theorem \ref{teo:existlin} the solution $f$ of \eqref{eq:systemlinear}, \eqref{eq:linbc} satisfies estimate \eqref{eq:estSol1} so that for each $j\in \{1,\dots,q\}$,  
\begin{align*}
& 
\| f_j\|_{C^{\frac{4+\alpha}{4},{4+\alpha}}([0,T]\times [0,1])} 
\le 
\sum_{i=1}^q \| f_i\|_{C^{\frac{4+\alpha}{4},{4+\alpha}}([0,T]\times [0,1])} 
\\ \nonumber
& \leq C_0 \Big( \sum_{i=1}^q ( \| R_i+h(\bar{f}_i)\|_{C^{\frac{\alpha}{4},\alpha}([0,T]\times [0,1])} + \|f_{0,i}\|_{C^{4,\alpha}([0,1])} 
+ |P_i|)  + \| b\|_{C^{0,\frac{1+\alpha}{4}}([0,T])} \Big) \, .
\end{align*}
with $C_0=C_0(n,q,\delta, \tilde{\delta})$. 
It then follows from 
applying triangle inequalities of H\"{o}lder-norms and \eqref{eq:vectorb} that 
for each $j\in \{1,\dots,q\}$,
\begin{align}\label{eq:estSol2}
& 
\| f_j\|_{C^{\frac{4+\alpha}{4},{4+\alpha}}([0,T]\times [0,1])} \\ \nonumber 
& \leq C_0  \sum_{i=1}^q \Big( \| R_i\|_{C^{\frac{\alpha}{4},\alpha}([0,T]\times [0,1])} + \| h(\bar{f}_i) - h(f_{0,i})\|_{C^{\frac{\alpha}{4},\alpha}([0,T]\times [0,1])}   \\ \nonumber
& \qquad + \|  (E_i-\bar{E}_i) \pa_x^3 \bar{f}_i\|_{C^{0,\frac{1+\alpha}{4}}([0,T])}
+ |\lambda |  \Big\| \frac{\pa_x \bar{f}_i}{|\pa_x \bar{f}_i|} -\frac{\pa_x f_{0,i}}{|\pa_x f_{0,i}|} \Big\|_{C^{0,\frac{1+\alpha}{4}}([0,T])} \Big) \\ \nonumber
& \qquad + 
C_0  \sum_{i=1}^q  \Big(\|h(f_{0,i})\|_{C^{0,\alpha}([0,1])} + \|f_{0,i}\|_{C^{4,\alpha}([0,1])} + |P_i|  + n|\lambda| \Big) \, .
\end{align}

The last term on the right hand side of \eqref{eq:estSol2} depends only on the initial data and $P_i,\lambda$, and will dictate the choice of the constant $M$. From the other terms on the right hand side we are able to gain a power of $T$ and hence to bound them choosing $T$ sufficiently small. Indeed, from \eqref{eq:R_i} using Lemmas \ref{lem:Holder}, \ref{lem:hilfsatz} and Remark \ref{rem:Holder} we find for $T \leq T_1$
\begin{align}\nonumber
\| R_i\|_{C^{\frac{\alpha}{4},\alpha}([0,T]\times [0,1])} 
& \leq C \Big\| \frac{1}{|\partial_x f_{0,i}|^4}-\frac{1}{|\partial_x \bar{f}_{i}|^4}\Big\|_{C^{\frac{\alpha}{4},\alpha}([0,T]\times [0,1])}  \|\pa_x^4 \bar{f}_{i}\|_{C^{\frac{\alpha}{4},\alpha}([0,T]\times [0,1])} \\ \label{eq:sm1}
& \leq C T^{\frac{\alpha}{4}} \, ,
\end{align}
with $C=C(n,\delta, \|\bar{f}_{i}\|_{C^{\frac{4+\alpha}{4},4+\alpha}}, \|f_{0,i}\|_{C^{4,\alpha}([0,1])})$. 
Furthermore, from \eqref{eq:gi} and using Lemmas \ref{lem:Holder}, \ref{lem:propHolder}, \ref{lem:delta}, \ref{lem:hilfsatz}, Remark \ref{rem:Holder} again for $T \leq T_1$, we find 
\begin{align}\nonumber
&  \| h(\bar{f}_i) - h(f_{0,i})\|_{C^{\frac{\alpha}{4},\alpha}([0,T]\times [0,1])} \\ \nonumber
& \leq C \sum_{k=1}^3 \| \pa_x^k \bar{f}_i - \pa_x^k f_{0,i}\|_{C^{\frac{\alpha}{4},\alpha}([0,T]\times [0,1])} 
+ C \sum_{k=1}^4  \left\| \frac{1}{|\partial_x f_{0,i}|^{2k}}-\frac{1}{|\partial_x \bar{f}_{i}|^{2k}}\right\|_{C^{\frac{\alpha}{4},\alpha}([0,T]\times [0,1])} \\ \label{eq:sm2}
& \leq C T^{\frac{\alpha}{4}} \, ,
\end{align}
with $C=C(n,\delta, \|\bar{f}_{i}\|_{C^{\frac{4+\alpha}{4},4+\alpha}}, \|f_{0,i}\|_{C^{4,\alpha}([0,1])})$. 
Similarly, for the boundary terms with Lemmas \ref{lem:Holder}, \ref{lem:delta}, \ref{lem:HolderChaos},  \ref{lem:propHolder}, \ref{lem:hilfsatz}, Remark \ref{rem:Holder} and \eqref{eq:estHolbdy} when $T \leq T_1$ we find 
\begin{align}\nonumber
\Big\| \frac{\pa_x \bar{f}_i}{|\pa_x \bar{f}_i|} -\frac{\pa_x f_{0,i}}{|\pa_x f_{0,i}|} \Big\|_{C^{0,\frac{1+\alpha}{4}}([0,T])} 
& \leq C \Big\| \frac{1}{|\pa_x \bar{f}_i|} -\frac{1}{|\pa_x f_{0,i}|} \Big\|_{C^{0,\frac{1+\alpha}{4}}([0,T])} \Big\| \pa_x \bar{f}_i \Big\|_{C^{0,\frac{1+\alpha}{4}}([0,T])} \\ \nonumber
& \quad  +  C \Big\| \frac{1}{|\pa_x f_{0,i}|} \Big\|_{C^{0,\frac{1+\alpha}{4}}([0,T])} \Big\| \pa_x \bar{f}_i - \pa_x f_{0,i} \Big\|_{C^{0,\frac{1+\alpha}{4}}([0,T])}  \\  \label{eq:sm3}
& \leq C T^{\frac{\alpha}{4}} \, ,
\end{align}
and finally by \eqref{eq:E_0-E_1}
\begin{equation} 
\label{eq:sm4}
\|  (E_i-\bar{E}_i) \pa_x^3 \bar{f}_i\|_{C^{0,\frac{1+\alpha}{4}}([0,T])} \leq C T^{\frac{\alpha}{4}},
\end{equation}
with 
$C=C(n,\delta, \|\bar{f}_{i}\|_{C^{\frac{4+\alpha}{4},4+\alpha}}, \|f_{0,i}\|_{C^{4,\alpha}([0,1])})$. 

From \eqref{eq:estSol2} together with \eqref{eq:sm1}, \eqref{eq:sm2}, \eqref{eq:sm3} and \eqref{eq:sm4} we obtain
\begin{align}\label{eq:estSol3}
& 
\| f_j\|_{C^{\frac{4+\alpha}{4},{4+\alpha}}([0,T]\times [0,1])} 
\\  \nonumber
& \leq C_2 T^{\frac{\alpha}{4}}+ 
C_0 \sum_{i=1}^q   \Big(\|h(f_{0,i})\|_{C^{0,\alpha}([0,1])} + \|f_{0,i}\|_{C^{4,\alpha}([0,1])} + |P_i|   + n|\lambda| \Big) \, ,
\end{align}
with 
$C_2=C_2(n,q,\delta, \tilde{\delta},\|f_{0}\|_{C^{4,\alpha}([0,1])},M)$ since $\|\bar{f}_i\|_{C^{\frac{4+\alpha}{4}, 4+\alpha}} \leq M$ by definition of $X_i$ for each $i \in \{1,2,\ldots, q\}$.

\paragraph{Proof of Theorem \ref{teo:STEpde}}

\begin{proof}[Proof of Theorem \ref{teo:STEpde}] 

\underline{Existence of a solution} We start by fixing $M$ and $T$. 
Let 
\begin{align}\label{defM}
 M:= 2  
 C_0  \sum_{i=1}^q  \Big( \|h(f_{0,i})\|_{C^{0,\alpha}([0,1])} + \|f_{0,i}\|_{C^{4,\alpha}([0,1])} + |P_i|   + n|\lambda| \Big) 
 \end{align}
see the last term in \eqref{eq:estSol3}. 
Now fix $T \leq \min\{T_1,1\}$, with $T_1$ defined in Lemma \ref{lem:delta}, such that 
$$C_1 T^{\frac{\alpha}{4}} < 1 \mbox{ and }C_2 T^{\frac{\alpha}{4}} < \frac{M}{2}\, ,$$
with $C_1,C_2$ the constants depending on $n$, $q$, $\delta$, $\tilde{\delta}$, $\|f_{0}\|_{C^{4,\alpha}[0,1]}$ and $ M$ appearing in \eqref{eq:Banach} and \eqref{eq:estSol3} respectively. 

Since the $\prod_{i=1}^q X_i$ is a closed set of the Banach space $C^{\frac{4+\alpha}{4}, 4+\alpha} \big([0,T]\times [0,1]; \R^{qn}\big)$ and, by the choice of $M$ and $T$, the map $\mathcal{R}$ is a self-map and a strict contraction, 
by applying Banach's fixed point theorem we get a unique fixed point of $\mathcal{R}$ and hence, by construction, a solution to \eqref{(P)} with \eqref{eq:nonlinearbc} and \eqref{eq:ic} in $C^{\frac{4+\alpha}{4}, 4+\alpha} \big([0,T]\times [0,1]; \R^{qn}\big)$. Moreover, $f_i$ is a regular curve for each $i\in \{1, \dots,q\}$ and $t \in [0,T]$ by Lemma \ref{lem:delta}.

\smallskip
\underline{Non-collinearity} 
 Consider the function
$$ \mbox{nc}: [0,T] \times [0,1] \to \R, \qquad \mbox{nc}(t,x)= 1- \prod_{1\leq i<j\leq q} \Big|\langle \pa_s f_{i},\pa_s f_{j} \rangle \Big|
\, , $$
for $f=(f_{1},\dots,f_{q})$ the solution provided so far.
The non-collinearity condition (NC) yields that at time $t=0$ and at $x=0$ this function is strictly positive. Moreover, the regularity and smoothness of the solution ensures that $\mbox{nc} \in C^{0}([0,T] \times [0,1])$ (see Remark \ref{rem:Holder}) and hence, by possibly choosing $T$ smaller, $\mbox{nc}>0$ on $[0,T] \times \{0\}$. In other words,
 the non-collinearity condition remains satisfied on $[0,T]$.

\smallskip
\underline{Uniqueness of the solution} Let $f,\tilde{f}$ be two different solutions on $[0,T] \times [0,1]$ of  \eqref{(P)} with \eqref{eq:nonlinearbc} and \eqref{eq:ic}.  Let 
$$ \bar{t}=\sup\{t\in [0,T]\, | \, f(\tau, x)=\tilde{f}(\tau, x) \quad \forall \, x \in [0,1] \text{ and } \forall \, \tau \leq t \}.$$
Obviously $0 \leq \bar{t} <T $. Now consider the problem \eqref{(P)}, \eqref{eq:nonlinearbc}, with initial data $f(\bar{t},\cdot)=\tilde{f}(\bar{t}, \cdot)$.  Note that for this new initial data (NC) is satisfied as well as all necessary compatibility conditions. By repeating the arguments provided so far, this problem has a unique solution on some small time interval $[\bar{t}, \bar{t} +\epsilon]$. Since $f$ and $\tilde{f}$ are also solutions, this yields that $f=\tilde{f}$ for some time after $\bar{t}$ giving a contradiction to the definition of $\bar{t}$.

\smallskip
\underline{Parabolic-Smoothing for positive time} The smoothness of the solution in $(0,T]\times[0,1]$ follows with similar arguments as presented in \cite[App.B.2.3]{DS}. For completeness we report here the main ideas. Given $0<\epsilon< T $, we consider the network $\gamma:=f\eta= (f_{1}\eta, \dots, f_{q}\eta)$ where $\eta:[0, T] \to [0,1]$ is some smooth  cut-off function with $\eta(t)=0$ for $  0 \leq t < \epsilon/4$ and $\eta(t) =1$ for $t \in [\epsilon/2 , T]$. By the regularity of $f$ it follows that $\gamma \in C^{\frac{4+\alpha}{4},4+\alpha}([0,T] \times [0,1])$. Moreover, upon recalling \eqref{(P)}, \eqref{eq:nonlinearbc}, \eqref{eq:ic}, and  \eqref{jay}, we infer that $\gamma$ satisfies the linear parabolic boundary value problem
\begin{align}\label{eq:systemgamma}
\partial_{t } \gamma_{i} =-\frac{1}{|\partial_{x} f_{i}|^{4}} \partial_{x}^{4}\gamma_{i} + \eta h(f_{i}) + f_{i} \frac{d}{dt} \eta \qquad i\in \{1,\dots,q\},
\end{align}
together with boundary conditions 
\begin{equation}\label{eq:nonlinearbcgamma}
\left\{\begin{aligned}
\gamma_i(t,1)& =\eta(t) P_i, &\mbox{ for all }t\in (0,T), i\in \{1,\dots,q\}, \\
\pa_x^2 \gamma_i(t,1) & =0=\pa_x^2 \gamma_i(t,0)  &\mbox{ for all }t\in (0,T), \ i\in \{1,\dots,q\},\\
\gamma_i(t,0) & =\gamma_j(t,0) &\mbox{ for all }t\in (0,T), i,j\in \{1,\dots,q\},\\
 \sum_{i=1}^q \big( E_{i}(f_{i}) &\partial_{x}^{3} \gamma_{i}(t,0) - \lambda_i \frac{1}{|\partial_{x} f_{i}|}\partial_x \gamma_{i}(t,0) \big) = 0  &\mbox{ for all }t\in (0,T),
\end{aligned} \right.
\end{equation}
with $E_{i}(f_{i})=\frac{1}{|\partial_{x}f_{i}|^{3}}(Id -  \partial_{s} f_{i} \otimes  \partial_{s} f_{i} ) $
 and initial datum $\gamma_{0, i} =0$, $i\in \{1,\dots,q\}$. Note that the system is linear and parabolic by regularity of $f_i$, $i\in \{1,\dots,q\}$. The compatibility conditions of any order are satisfied (thanks to $\gamma$ being identically zero close to the origin) and the complementary conditions are also satisfied (this is done in a similar way as in the previous section and exploiting the fact that (NC) holds for all times in $[0,T]$). The coefficients of the elliptic and boundary operators  belong to $C^{\frac{3+\alpha}{4}, 3+\alpha}([0,T] \times [0,1])$ resp. to $C^{0, \frac{3+\alpha}{4}} ([0,T])$ whereas the inhomogeneity in \eqref{eq:systemgamma} is in $C^{\frac{1+\alpha}{4}, 1+\alpha}([0,T] \times [0,1])$. Application of \cite[Thm.4.9, page 121]{Sol}  yields $\gamma \in  C^{\frac{5+\alpha}{4}, 5+\alpha}([0,T] \times [0,1])$ and therefore $f \in  C^{\frac{5+\alpha}{4}, 5+\alpha}([\epsilon/2,T] \times [0,1])$. To apply a bootstrapping argument we now repeat the same procedure, but since the higher regularity of $f$ is guaranteed only  for $t \geq \epsilon/2$  the next cutting function must be zero, say on $[0, \frac{2}{3} \epsilon]$ and equal one on $[\frac{3}{4}\epsilon, T]$, i.e. we have to ``shift and reduce'' progressively the interval where $\eta \in (0,1)$.  More details  in this respect can be found in \cite[App.B.2.3]{DS}. Eventually we attain $f \in C^{\infty} ([\epsilon, T])$ and since $\epsilon$ was arbitrarily chosen the claim follows.
\end{proof}

\subsection{The case of a smooth initial datum}

If the initial data $f_{0,i}$, $i\in \{1,\dots,q\}$, are in $C^{k,\alpha}([0,1])$, $k\geq 4$, and higher order compatibility conditions are satified we get a solution with higher regularity. Let us first state the compatibility conditions of general order (see Remark \ref{compaorder1} for compatibility conditions of order zero).

\begin{rem}\label{rem2.1}
(\textbf{Compatibility conditions analytical problem})
Following \cite[page~98]{Sol} and \cite[page 217, Example 6.12]{EZ}, 
for the problem \eqref{(P)}, \eqref{eq:nonlinearbc}, with initial datum $f_{0}$, 
we say that compatibility conditions of order $\mu \in \mathbb{N} \cup \{ 0 \}$
are satisfied if the following hold:
\begin{itemize}
\item we have
\begin{align*}
 f_{0,i}(1) =P_i, \; \text{ and }\;\,\,  f_{0,i}(0)- f_{0,j}(0)=0, \; \quad \ i,j \in \{1,\dots,q\},
\end{align*}
\item for any $i_{q} \in  \mathbb{N}$ such that $4 i_{q} -4 \leq \mu$ we have
\begin{align*}
\partial^{i_{q}}_{t} f_{i} \Big|_{(t,x)=(0,1)}=0, \; \; \text{ and }\;  (\partial_{t}^{i_{q}} f_{i}-\partial_{t}^{i_{q}} f_{j})\Big|_{(t,x)=(0,0)}=0 , \;  i,j \in \{1,\dots,q\},
\end{align*}
\item for any $i_{q} \in  \mathbb{N} \cup \{ 0 \}$ such that $4 i_{q} -2 \leq \mu$ we have
\begin{align*}
\partial^{i_{q}}_{t} \left( 
\partial^{2}_{x} f_{i}\right) \Big|_{(t,x)=(0,1)}=0, \; \text{ and }\;  \partial^{i_{q}}_{t} \left( 
\partial^{2}_{x} f_{i} \right)\Big|_{(t,x)=(0,0)}=0, \; i\in \{1,\dots,q\},
\end{align*}
\item for any $i_{q} \in  \mathbb{N} \cup \{ 0 \}$ such that $4 i_{q} -1 \leq \mu$ we have
\begin{align*}
\partial^{i_{q}}_{t} \left(\sum_{i=1}^q (\nabla_s \vec{\kappa}_{i}(t,0) - \lambda_i \partial_s f_{i}(t,0)) \right) \Big|_{t=0} = 0.
\end{align*}
\end{itemize}
The above conditions should be understood as follows: upon recalling \eqref{(P)}, \eqref{eq:gi}, and \eqref{varphi*}, let $L^{*}_{i}$, $i\in\{1,\dots,q\}$, be the differential operator such that
\begin{align*}
L^{*}_{i}f_{i}& = \partial_{t}f_{i}=-\frac{1}{|\pa_x f_i|^4}\pa_{x}^4 f_i+h(f_i)\\
&=- \nabla_s^2 \vec{\kappa_i}-\frac12 |\vec{\kappa}_i|^2 \vec{\kappa}_i + \lambda_{i} \vec{\kappa}_i + \varphi_{i}^{*} \partial_{s} f_{i}
\end{align*} 
and let $$L^{* (i_{q})}_{i} f_{i} = \partial_{t}^{i_{q}} f_{i},$$ where one can use \cite[Lemmas~3.1, 3.5, 3.6]{DLPnetwork1} to derive an explicit expression for $\partial_{t}^{i_{q}} f_{i}$  free of time derivatives. Then the first condition can be rephrased as
\begin{align*}
L^{* (i_{q})}_{i} f_{0,i}&=0 \text{ at } x=1 \text{ for } i\in \{1,\dots,q\}, \text{ and }\\
L^{* (i_{q})}_{i} f_{0,i}&=L^{* (i_{q})}_{j} f_{0,j } \text{ at } x=0 \text{ for } i\neq j.
\end{align*}
The other conditions are understood in a similar way.
For instance the second set of conditions can be rephrased as
\begin{align*}
\pa_{x}^{2} L^{* (i_{q})}_{i} f_{0,i}&=0 \text{ at } x \in \{0,1 \} \text{ for } i\in \{1,\dots,q\}.
\end{align*}
\end{rem}

\begin{teo}\label{teo:STEpdesmooth}
Let $n \geq 2$, $q\geq 3$, 
$\alpha\in(0,1)$, $k\in \N$, $k\geq 4$ and $P_i$, $i\in \{1,\dots,q\}$, be  points in $\R^n$.  
Given $f_{0,i}:[0,1] \to \R^n$, $f_{0,i} \in C^{k,\alpha}([0,1])$, $i\in\{1,\dots,q\}$, regular maps 
satisfying the compatibility conditions of order $(k-4)$ 
(as stated in Remark~\ref{rem2.1}) and the non-collinearity condition (NC), then there exist $T>0$ and 
regular curves $f_{i} \in C^{\frac{k+\alpha}{4},k+\alpha}([0,T]\times [0,1];\R^n)$, $i\in\{1,\dots,q\}$, 
 such that $f=(f_1,\dots,f_q)$ is the unique solution of \eqref{(P)} together with the boundary conditions \eqref{eq:nonlinearbc} and the initial condition $f_i (t=0)=f_{0,i}$.

Moreover, we have instant parabolic smoothing, that is  $f_i \in C^{\infty}((0,T]\times [0,1])$ for any $i\in\{1,\dots,q\}$ and the non-collinearity condition (NC) holds at the triple junction for any time $t \in [0,T]$.
\end{teo}
\begin{proof}
Since the assumption of Theorem \ref{teo:STEpde} are satisfied, there exist $T>0$ and regular curves $f_{i} \in C^{\frac{4+\alpha}{4},4+\alpha}([0,T]\times I;\R^n) \cap  C^{\infty}((0,T]\times [0,1])$, $i\in\{1,\dots,q\}$, 
 such that $f=(f_1,\dots,f_q)$ is the unique solution of \eqref{(P)} satisfying the boundary conditions \eqref{eq:nonlinearbc} and the initial condition. Moreover the solution satisfies the non-collinearity condition on $[0,T]$. It remains to show that, in case $k\geq 5$, the solution is actually more regular. We observe that $f_i$ for $i\in\{1,\dots,q\}$ solve the \textit{linear} PDE system
$$\pa_t f_i = - a_i \pa_x^4 f_i +b_i \mbox{ in }(0,T]\times (0,1), \qquad i\in\{1,\dots,q\},$$
with boundary conditions
\begin{equation*}
\left\{\begin{array}{ll}
f_i(t,x=1) =P_i , \qquad f_i(t,x=0)=f_j(t,x=0)  \mbox{ }&\forall \, t \in [0,T], \ i,j\in\{1,\dots,q\}\\
\pa_x^2 f_i=0 \mbox{ }&\forall \, t \in [0,T], \ x=0,1, \ i\in\{1,\dots,q\}\\
\sum_{i=1}^q c_i \pa_x^3 f_i = \sum_{i=1}^q  q_i   \mbox{ }&\forall \, t, \ x=0,\end{array} \right.
\end{equation*}
and initial condition $f_i(t=0) =f_{i,0}$ on $[0,1]$, $i\in\{1,\dots,q\}$, by looking at the non-linear initial boundary value problem satisfied by $f_i$ as a linear problem for $f_i$ with given coefficients (since we already have a solution: recall \eqref{(P)}, \eqref{eq:gi}, \eqref{jay}). The coefficients satisfy $a_i \in C^{\frac{3+\alpha}{4},3+\alpha}([0,T]\times [0,1])$, $b_i \in C^{\frac{1+\alpha}{4},1+\alpha}([0,T]\times [0,1])$ and $c_i \in C^{\frac{3+\alpha}{4},3+\alpha}([0,T]\times [0,1])$ and the boundary data satisfies $q_i \in C^{\frac{3+\alpha}{4},3+\alpha}([0,T]\times [0,1])$. The system is parabolic by the regularity of $f_i$, $i\in\{1,\dots,q\}$, and by the assumptions on the initial datum the compatibility conditions of order zero and one are satisfied. Proceeding similarly as in the previous section \ref{secLP} one shows that the complementary conditions are satisfied since the non-collinearity condition is satisfied on $[0,T]$. By the regularity of the initial datum and \cite[Thm.4.9, page 121]{Sol} we find $f_i \in  C^{\frac{5+\alpha}{4}, 5+\alpha}([0,T] \times [0,1])$, $i\in\{1,\dots,q\}$. Being the solution more regular, we can repeat the argument as long as the smoothness of the initial datum and the order of the compatibility condition allow.
\end{proof}

By the previous result we immediately infer an existence result in $C^{\infty}$.

\begin{cor}\label{teo:STEpdeCinfty}
Let $n \geq 2$, $q\geq 3$, 
and $P_i$, $i\in\{1,\dots,q\}$, be points in $\R^n$.  
Given $f_{0,i}:[0,1] \to \R^n$, $f_{0,i} \in C^{\infty}([0,1])$, $i\in\{1,\dots,q\}$, regular maps satisfying the compatibility conditions of any order (as stated in Remark~\ref{rem2.1}) and the non-collinearity condition (NC), then there exist $T>0$ and 
regular curves $f_{i} \in C^{\infty}([0,T]\times [0,1];\R^n)$, $i\in\{1,\dots,q\}$, 
 such that $f=(f_1,\dots,f_q)$ is the unique solution of \eqref{(P)} together with the boundary conditions \eqref{eq:nonlinearbc} and the initial condition $f_i (t=0)=f_{0,i}$. Moreover, the non-collinearity condition (NC) holds at the triple junction for any time $t \in [0,T]$.
\end{cor}
\begin{proof}
By the arguments in the proof of Theorem \ref{teo:STEpdesmooth} one sees that the time interval of existence of the solution is independent of $k$. This immediately yields the result.
\end{proof}

\section{Solution to the geometrical problem}\label{sec:equiv}

\begin{proof}[Proof of Theorem \ref{teo:STEgeo}]
Let $f_0=(f_{0,1},f_{0,2},\ldots, f_{0,q})$ be as in the statement  and set 
\begin{equation}\label{eq:repar}
\phi_{i}:[0,1] \to [0,1],  \quad \phi_{i}(x)=\frac{1}{\mathcal{L} (f_{0,i})} \int_{0}^{x} |\partial_{x} f_{0,i}| dx, 
\end{equation}
to be reparametrizations so that $\tilde{f}_{0,i}(x):= f_{0,i}(\phi^{-1}_{i}(x))$, $i=1,2,\ldots, q$, are now parametrized by constant speed. 
By Remark \ref{rem:reparid} we have $\tilde{f}_{0,i} \in C^{4,\alpha}([0,1],\R^n)$.   
Then \eqref{eq:nonlinearbcgeo2}, \eqref{eq:ccfirststep-zero}, \eqref{eq:ccfirststep} and Remark \ref{rem:Acs} imply that
\begin{align*}
&\pa_x^2 \tilde{f}_{0,i}=0 \mbox{ at }x=0,1 , \qquad 
\frac{1}{|\pa_x \tilde{f}_{0,i}|^4}\pa_x^4 \tilde{f}_{0,i}=0 \mbox{ at }x=1\\
& \mbox{and }\; 
\frac{1}{|\pa_x \tilde{f}_{0,i}|^4}\pa_x^4 \tilde{f}_{0,i} = \frac{1}{|\pa_x \tilde{f}_{0,j}|^4}\pa_x^4 \tilde{f}_{0,j} \mbox{ at }x=0 \mbox{ for }i \ne j
\end{align*}
and \eqref{eq:nonlinearbcid} hold. In other words the Compatibility Conditions \ref{compaorder1} are fulfilled
and $\tilde{f}_{0}$  is an admissible initial datum for the analytical problem.  
 Theorem~\ref{teo:STEpde}
yields the existence of $T>0$ and $\tilde{f}_{i} \in C^{\frac{4+\alpha}{4}, 4 +\alpha} ([0,T] \times [0,1]; \R^{n}) \cap C^{\infty}((0,T) \times [0,1])$ solutions to \eqref{(P)} (hence of \eqref{eq:flowgeomtang} with tangential components specified in \eqref{varphi*}) together with the initial condition $\tilde{f}_{0}$ and boundary conditions \eqref{eq:nonlinearbc}, that is we have found a solution for the geometric problem with initial datum $\tilde{f}_{0}=f_{0}\circ \phi^{-1}$, a reparametrization of $f_{0}$.
\end{proof}

\begin{proof}[Proof of Theorem \ref{teo:STEgeoCinfty}]
The statement is a direct consequence of Corollary~\ref{teo:STEpdeCinfty}.
\end{proof}


We now turn to the question of geometric uniqueness.
\begin{lemma}[Geometric uniqueness]\label{geomuniq}
Suppose that, given a smooth initial network $f_{0}$ satisfying the assumptions of  Theorem \ref{teo:STEgeoCinfty},  we have two smooth solutions $f=(f_{1},f_{2},\ldots, f_{q})$ and $\bar{f}=(\bar{f}_1,\bar{f}_2,\ldots, \bar{f}_q)$ (in the sense of Theorem~\ref{teo:STEgeoCinfty}) defined on 
$[0,T]\times [0,1]$ and on $[0,\bar{T} ]\times [0,1]$ respectively. Then the sets $\Gamma(t)=\{ (f_{1}(t,x), \ldots, f_{q}(t,x)) \,: \, x \in [0,1] \}$ and $\bar{\Gamma}(t)=\{ (\bar{f}_{1}(t,x), \ldots, \bar{f}_{q}(t,x)) \,: \, x \in [0,1] \}$ coincide for some small time.
\end{lemma}
\begin{proof}
Let $\tilde{f}_0$ be the initial datum reparametrized by constant speed and $\tilde{f}$, defined on $[0,\tilde{T} ]\times I$, be the analytical solution given by Corollary~\ref{teo:STEpdeCinfty} with initial datum $\tilde{f}_0$. This is also a solution for the geometric problem (by the proof of Theorem \ref{teo:STEgeoCinfty}). For $f$ as in the statement, it is enough to show that the sets $f(t)$ and $\tilde{f}(t)$ coincide  on a subset $[0,T_{e}] \subset [0, \min\{T, \tilde{T}\}]$.
Since the set $f(t)$ is invariant under reparametrization of the maps describing it, and since every map considered is smooth, without loss of generality we can assume that $f$ is parametrized by constant speed (cf. \eqref{eq:repar} for a similar argument).  Thus each $f_i$, $i \in \{1, ..,q\}$, solves \eqref{eq:flowgeomtang} with some smooth tangential component $\varphi_i$, together with the boundary conditions \eqref{eq:nonlinearbcgeo} and initial datum $\tilde{f}_{0,i}$. Note that due to the constant speed parametrization the boundary conditions $\vec{\kappa}_{i}=0$ are equivalent to $\partial_{x}^{2} f_{i} =0$ for any $i \in \{1, \dots, q\}$.
The solution $\tilde{f}_i$, on the other hand, solves \eqref{eq:flowgeomtang} with tangential components $\varphi_{i}^{*}$ as in \eqref{varphi*}, boundary conditions \eqref{eq:nonlinearbc} and  initial datum $\tilde{f}_0$.
The proof is then complete if we show the existence of smooth diffeomorphisms $\phi_{i}$, $i=1, \ldots, q$ such that $\tilde{f}_{i}(t,x)= f(t, \phi_{i}(t,x))$ for $t \in [0,T_{e}]$ for some $0<T_{e} \leq \min\{T, \tilde{T}\}$.
Suppose first that such diffeomeorphisms exist. Then using the flow equations we compute
\begin{align*}
\partial_{t} \tilde{f}_{i}(t,x) &= \partial_{t} f_{i} (t, \phi_{i}(t,x)) + \pa_{x}f_{i} (t, \phi_{i}(t,x)) \partial_{t} \phi_{i} (t,x) \\
&=\left[- \nabla_s^2 \vec{\kappa}_{i}-\frac12 |\vec{\kappa}_{i}|^2 \vec{\kappa}_{i} + \lambda_{i} \vec{\kappa}_{i} + \varphi_{i} \partial_{s} f_{i} \right] (t, \phi_{i}(t,x)) + \pa_{x}f_{i} (t, \phi_{i}(t,x)) \partial_{t} \phi_{i} (t,x)\\
 &= \left[- \nabla_s^2 \vec{\tilde{\kappa}}_{i}-\frac12 |\vec{\tilde{\kappa}}_{i}|^2 \vec{\tilde{\kappa}}_{i} + \lambda_{i} \vec{\tilde{\kappa}}_{i} +\varphi_i^{*} \pa_s \tilde{f}_i \right] (t,x)
\\
& \quad + \big[\varphi_{i}(t, \phi_{i}(t,x)) -\varphi_i^{*}(t,x)  + | \pa_{x}f_{i} (t, \phi_{i}(t,x))| \partial_{t} \phi_{i} (t,x) \big]\partial_{s} f_{i} (t, \phi_{i}(t,x)) \, .
\end{align*}
It follows that each diffeomorphism $\phi_i$, $i\in \{1,..,q\}$, has to solve the first order ODE
\begin{equation}\label{eq:ODEdiff}
\partial_{t} \phi_{i} (t,x) = \frac{1}{| \pa_{x}f_{i} (t, \phi_{i}(t,x))|} \left( -\varphi_{i}(t, \phi_{i}(t,x))+\varphi_i^{*}(t,x)\right)\, , 
\end{equation}
with initial datum 
$\phi_i(0,x)=x$ for each $x\in [0,1]$. 
The right hand side in \eqref{eq:ODEdiff} can be written as $G(t, x,\phi_{i}(t,x))$, with $G(t,x,y)$ a smooth functions in its variables. Here $x$ plays the role of a parameter, $x\in [0,1]$. 
Since $\varphi_{i}(t,x)=\varphi_i^{*}(t,x)$ for $x \in \{0,1\}$ and all $t$ by Remark~\ref{rem:4.1}, 
we see from \eqref{eq:ODEdiff} that  $\phi_{i}(t,x)=x$, for all $t$ and $x\in\{0,1\}$. 
The existence, uniqueness and regularity of the solution follow from \cite[Sec.1.3]{Mantegazza} and \cite[Chap.9 and App.D]{Lee}. 
The smoothness of the solution together with the assumption on the initial datum $\pa_x \phi_i=1$ imply that also $\pa_{x}\phi_{i} >0$ on $[0,1]$ is satisfied for some small time. By existence of these diffeomorphisms, the claim follows.
\end{proof}

\renewcommand{\thesection}{}

\appendix\renewcommand{\thesection}{\Alph{section}}
\setcounter{equation}{0}
\renewcommand{\theequation}{\Alph{section}\arabic{equation}}

\section{Supporting materials}\label{secA}

Here we collect some useful formulas.
For $f: I \to \R^{n}$ a regular parametrization of a curve and for sufficiently smooth $\phi:I \to \R^n$ the first variation of the length is given by 
\begin{equation}
\label{dlength}
\frac{d}{d\varepsilon} \mathcal{L}(f+\varepsilon \phi) \Big|_{\varepsilon=0}= \frac{d}{d\varepsilon}  \int_I   |\partial_x (f+\varepsilon \phi)| dx  \Big|_{\varepsilon=0}=   \langle \partial_s f , \phi \rangle \Big|_{\partial I} - \int_{I} \langle \vec{\kappa} ,  \phi \rangle \, ds \, ,
\end{equation} 
while the first variation of elastic energy (see \cite[Proof of Lemma A1]{DP}) is
\begin{align}\nonumber
\frac{d}{d\varepsilon} \mathcal{E}(f+\varepsilon \phi) \Big|_{\varepsilon=0} & = \frac{d}{d\varepsilon}  \int_I  |\vec{\kappa}_{f+\varepsilon \phi}|^2  |\partial_x (f+\varepsilon \phi)| dx  \Big|_{\varepsilon=0} \\
& \hspace{-1cm}=   \langle \partial_s \phi ,\vec{\kappa} \rangle \Big|_{\partial I}  -  \langle  \phi ,\nabla_s \vec{\kappa} +\frac12  |\vec{\kappa}|^2 \partial_s f \rangle \Big|_{\partial I} + \int_{I} \langle \nabla_s^2 \vec{\kappa} +\frac12 |\vec{\kappa}|^2 \vec{\kappa} , \phi \rangle \, ds  \, .
\label{delastic}
\end{align}
Moreover, 
\allowdisplaybreaks{\begin{align}\nonumber
\vec{\kappa}= \partial_{s}^2 f &=\frac{\pa_x^2 f}{|\pa_x f|^{2}} - \langle \pa_x^2 f, \pa_x f \rangle \frac{\pa_x f}{|\pa_x f|^{4}}=\frac{\pa_x^2 f}{|\pa_x f|^{2}} - \langle \frac{\pa_x^2 f}{|\pa_x f|^{2}}, \partial_{s}  f\rangle \partial_{s} f,\\
\nonumber
|\vec{\kappa}|^{2} &=\frac{|\pa_x^2 f|^{2}}{|\pa_x f|^{4}} - \frac{(\langle \pa_x^2 f, \pa_x f \rangle)^{2}}{|\pa_x f|^{6}}, \\ \nonumber
\partial_{s}\vec{\kappa}= \partial_{s}^3 f &= \frac{\pa_x^3 f}{|\pa_x f|^{3}} - \langle \frac{\pa_x^3 f}{|\pa_x f|^{5}}, \pa_x f \rangle \pa_x f \\ \nonumber
& \quad -3 \frac{\pa_x^2 f}{|\pa_x f|^{5}} \langle \pa_x^2 f, \pa_x f \rangle + 4 \frac{(\langle \pa_x^2 f, \pa_x f \rangle)^{2}}{|\pa_x f|^{7}} \pa_x f - \frac{|\pa_x^2 f|^{2}}{|\pa_x f|^{5}} \pa_x f\\ \nonumber
 &= \frac{\pa_x^3 f}{|\pa_x f|^{3}} - \langle \frac{\pa_x^3 f}{|\pa_x f|^{3}}, \pa_s f \rangle \pa_s f \\ \label{eq:Aderkappa}
& \quad -3 \vec{\kappa}  \langle \frac{\pa_x^2 f}{|\pa_x f|^{2}},\pa_s f \rangle 
+  \frac{(\langle \pa_x^2 f, \pa_x f \rangle)^{2}}{|\pa_x f|^{6}} \pa_s f - \frac{|\pa_x^2 f|^{2}}{|\pa_x f|^{4}} \pa_s f,\\ \nonumber
\langle \vec{\kappa}, \partial_{s} \vec{\kappa} \rangle &= \frac{\langle \pa_x^3 f, \pa_x^2 f \rangle}{|\pa_x f|^{5}}
- \frac{\langle \pa_x^3 f, \pa_x f \rangle}{|\pa_x f|^{7}}\langle \pa_x f, \pa_x^2 f \rangle - 3 \frac{|\pa_x^2 f|^{2}}{|\pa_x f|^{7}} \langle \pa_x f, \pa_x^2 f \rangle  + 3 \frac{ (\langle \pa_x f, \pa_x^2 f \rangle )^{3} }{|\pa_x f|^{9}}
,\\ \nonumber
\partial_{s}^2\vec{\kappa}= \pa_s^4 f &= \frac{\pa_x^4 f}{|\pa_x f|^{4}} 
- 6 \langle \pa_x^2 f, \pa_x f \rangle \frac{\pa_x^3 f}{|\pa_x f|^{6}} \\ \nonumber
& \quad - 4 \frac{|\pa_x^2 f|^{2}}{|\pa_x f|^{6}} \pa_x^2 f 
- 4 \frac{\pa_x^2 f}{|\pa_x f|^{6}}\langle \pa_x^3 f, \pa_x f \rangle 
+ 19 \pa_x^2 f \frac{(\langle \pa_x^2 f, \pa_x f \rangle)^{2}}{|\pa_x f|^{8}}\\ \nonumber
& \quad + \frac{\pa_x f}{|\pa_x f|} \Big{[}  - \langle  \frac{\pa_x^4 f}{|\pa_x f|^{5}}, \pa_x f \rangle  
- 3 \langle  \frac{\pa_x^3 f}{|\pa_x f|^{5}}, \pa_x^2 f\rangle + 13  \langle\frac{\pa_x^3 f}{|\pa_x f|^{7}}, \pa_x f \rangle \langle \pa_x^2 f, \pa_x f \rangle \\ \nonumber
& \qquad + 13  \langle\frac{\pa_x^2 f}{|\pa_x f|^{7}}, \pa_x f \rangle | \pa_x^2 f|^{2} -28\frac{(\langle \pa_x^2 f, \pa_x f \rangle)^{3}}{|\pa_x f|^{9}} \Big{]}.
\end{align}}
In particular it follows for the  velocity in \eqref{eq:flowgeomtang}
\begin{align}\nonumber
\partial_{t}f  &= -\nabla_{s}^{2} \vec{\kappa} -\frac{1}{2} |\vec{\kappa}|^{2} \vec{\kappa} + \lambda \vec{\kappa} + \varphi \partial_{s} f \\ \nonumber
&= -\pa_s^2 \vec{\kappa} -3 \langle \partial_{s} \vec{\kappa}, \vec{\kappa} \rangle \partial_{s} f -\frac{3}{2}|\vec{\kappa}|^{2} \vec{\kappa} + \lambda \vec{\kappa} + \varphi \partial_{s} f\\ \nonumber
& = -\frac{\pa_x^4 f}{|\pa_x f|^{4}} + 6 \langle \pa_x^2 f, \pa_x f \rangle \frac{\pa_x^3 f}{|\pa_x f|^{6}} \\ \nonumber
& \quad +
\frac{\pa_x^2 f}{|\pa_x f|^{2}} \Big{(} \frac{5}{2} \frac{|\pa_x^2 f|^{2}}{|\pa_x f|^{4}}  
+ 4 \frac{\langle \pa_x^3 f, \pa_x f \rangle}{|\pa_x f|^{4}} 
- \frac{35}{2} \frac{(\langle \pa_x^2 f, \pa_x f \rangle)^{2}}{|\pa_x f|^{6}} +\lambda\Big{)}\\ \nonumber
& \quad -\frac{\pa_x f}{|\pa_x f|} \Big{[} -  \langle \frac{\pa_x^4 f}{|\pa_x f|^{5}}, \pa_x f \rangle + 10  \frac{\langle \pa_x^2 f, \pa_x f \rangle}{|\pa_x f|^{7}} \langle \pa_x^3 f, \pa_x f \rangle  +\frac{5}{2} \langle \pa_x^2 f, \pa_x f \rangle \frac{|\pa_x^2 f|^{2}}{|\pa_x f|^{7}} \\ \label{eq:Aflowexpl}
& \quad 
-\frac{35}{2}\frac{(\langle \pa_x^2 f, \pa_x f \rangle)^{3}}{|\pa_x f|^{9}} 
+ \lambda \frac{\langle \pa_x^2 f, \pa_x f \rangle}{|\pa_x f|^{3}}  -\varphi
\Big{]}.
 \end{align}

\begin{rem}\label{rem:Acs}
Let $f:[0,1] \to \R^n$ be a regular sufficiently smooth curve parametrized by constant speed (equal to its length), that is $|\pa_x f|\equiv L(f)$ on $[0,1]$. Then
\begin{align*}
& \langle \pa_x f, \pa_x^2 f\rangle =0, \; \qquad |\pa_x^2 f|^2+ \langle \pa_x f, \pa_x^3 f\rangle =0,\\
& \mbox{and } \qquad 3 \langle \pa_x^2 f, \pa_x^3 f\rangle +\langle \pa_x f, \pa_x^4 f\rangle =0 \, ,
\end{align*}
so that (with similar calculations as in \eqref{eq:Aflowexpl}) we immediately obtain
\begin{align*}
\nabla_s^2 \vec{\kappa} & = \frac{1}{|\pa_x f|^4} \pa_x^4 f - \frac{1}{|\pa_x f|^6} \langle \pa_x^4 f, \pa_x f\rangle \pa_x f
+\frac{1}{|\pa_x f|^6} |\pa_x^2 f|^2 \pa_x^2 f\\
& = \frac{1}{|\pa_x f|^4} \pa_x^4 f +\frac{3}{|\pa_x f|^6}  \langle \pa_x^2 f, \pa_x^3 f\rangle \pa_x f
+\frac{1}{|\pa_x f|^6} |\pa_x^2 f|^2 \pa_x^2 f.
\end{align*}
\end{rem}

\section{Function spaces}\label{sec:fs}

Our short-time existence theory uses parabolic H\"older spaces, which are defined as  follows.
Following \cite[page 66]{Sol}, for a function $v : [0,T] \times [0,1] \to \R$ and $\rho \in (0,1)$ let
\begin{align*}
[v]_{\rho,x}&:= \sup_{(t,x), (t,y) \in [0,T] \times [0,1]} \frac{|v(t,x)-v(t,y) |}{|x-y|^{\rho}},\\
[v]_{\rho,t}&:= \sup_{(t,x), (t',x) \in [0,T] \times [0,1]} \frac{|v(t,x)-v(t',x) |}{|t-t'|^{\rho}}.
\end{align*}
As in \cite[pages 91 and 66]{Sol} we define 
$$C^{ \frac{k+\alpha}{4},k+\alpha} ([0,T] \times [0,1]) \qquad \text{ for } \alpha \in (0,1) \text{ and } k \in \mathbb{N}_0$$
to be the space of all maps $v : [0,T] \times [0,1] \to \R$ with continuous derivatives $\partial_{t}^{i}\partial_{x}^{j}v$ for  $i, j \in \N \cup \{0 \}$ with $4i+j \leq k$ and such that the norm 
\begin{align*}
\|v \|_{ C^{ \frac{k+\alpha}{4}, k+\alpha}([0,T] \times [0,1]) } &:= \sum_{4i+j=0}^{k} \sup_{(t,x) \in [0,T] \times [0,1]} |\partial_{t}^{i}\partial_{x}^{j}v (t,x)| \\
& \quad+ \sum_{4i+j=k} [\partial_{t}^{i}\partial_{x}^{j}v ]_{\alpha,x} + \sum_{0<k+\alpha-4i-j <4}  [\partial_{t}^{i}\partial_{x}^{j}v ]_{\frac{k+\alpha -4i-j}{4},t}
\end{align*}
is finite. Notice that in the last term we sum over $i,j$'s satisfying the inequality. In the proofs, in order to avoid lengthy notation, we do not write the set when considering the parabolic H\"older spaces. That is we write simply $\|v \|_{C^{\frac{k+\alpha}{4}, k+\alpha}}$ instead of $\|v \|_{ C^{ \frac{k+\alpha}{4}, k+\alpha}([0,T] \times [0,1]) }$. When considering the H\"older norms in only one variable we always write the set, for instance in $C^{4,\alpha}([0,1])$ or $C^{0,\frac{\alpha}{4}}([0,T])$. 

\emph{When dealing with vector-valued maps we use the convention that the  $C^{\frac{k+\alpha}{4},k+\alpha }$-norm of the vector is the sum of the norms of its components. }


\begin{rem}\label{rem:Holder}
From the definition it follows that for $m\leq k$, $m,k \in \mathbb{N}_0$
$$  C^{ \frac{m+\alpha}{4},m+\alpha}([0,T]\times[0,1]) \subset  C^{ \frac{k+\alpha}{4}, k+\alpha}([0,T]\times[0,1]) \, ,$$
and if $v \in C^{ \frac{k+\alpha}{4}, k+\alpha}([0,T]\times[0,1])$, then $\pa_x^l v \in C^{ \frac{k-l+\alpha}{4}, k-l+\alpha}([0,T]\times[0,1])$ for all $0 \leq l\leq k$ so that
$$\|\pa_x^l v \|_{C^{ \frac{k-l+\alpha}{4}, k-l+\alpha}([0,T]\times[0,1])} \leq \| v\|_{C^{ \frac{k+\alpha}{4}, k+\alpha}([0,T]\times[0,1])} \, .$$
In particular at each fixed $x \in [0,1]$ we have $\pa_x^l v (\cdot,x )\in C^{s,\beta}([0,T])$ 
with $s =[\frac{k-l+\alpha}{4}]$ and $\beta = \frac{k-l+\alpha}{4}-s$.
\end{rem}

We will use often the following properties of the H\"older norms. 
\begin{lemma}\label{lem:Holder}
For $k \in \mathbb{N}_0$, $\alpha,\beta \in (0,1)$ and $T>0$ we have
\begin{enumerate}
\item if $v,w \in C^{\frac{k+\alpha}{4},k+\alpha}([0,T]\times [0,1])$, then
$$ \| v w \|_{C^{\frac{k+\alpha}{4},k+\alpha}} \leq C \| v \|_{C^{\frac{k+\alpha}{4},k+\alpha}} \| w \|_{C^{\frac{k+\alpha}{4},k+\alpha}} \, ,$$
with $C=C(k)>0$;
\item if $v\in C^{\frac{\alpha}{4},\alpha}([0,T]\times [0,1])$, $v(t,x)\ne 0$ for all $(t,x)$, then
$$ \Big\| \frac{1}{v} \Big\|_{C^{\frac{\alpha}{4},\alpha}} \leq \Big\| \frac{1}{v} \Big\|^2_{C^{0}([0,T]\times [0,1])} \| v \|_{C^{\frac{\alpha}{4},\alpha}} \, .$$
\end{enumerate}
Similar statements are true for functions in $C^{k,\beta}([0,T])$ and $C^{k,\beta}([0,1])$.
\end{lemma}
\begin{proof}
It follows by the definition of the norms and direct computation.
\end{proof}

\begin{lemma}\label{lem:HolderChaos}
For $n \in \mathbb{N}$, $k \in \mathbb{N}_0$, $\alpha,\beta \in (0,1)$ and $T>0$ we have
\begin{enumerate}
\item 
if a vector-field $v\in C^{\frac{\alpha}{4},\alpha}([0,T]\times [0,1];\R^n)$, then
$$ \|\, |v|\, \|_{C^{\frac{\alpha}{4},\alpha}} \leq C  \| v \|_{C^{\frac{\alpha}{4},\alpha}} \, ,$$
with $C=C(n)$.
\item for $v, w\in C^{\frac{\alpha}{4},\alpha}([0,T]\times [0,1];\R^n)$  we have
$$ \|\, |v|- |w|\, \|_{C^{\frac{\alpha}{4},\alpha}} \leq  C  \left \| \frac{1}{|v|+|w|} \right\|^{2}_{C^{0} ([0,T]\times [0,1])}
 (\|  v \|_{C^{\frac{\alpha}{4},\alpha}} + \|  w\|_{C^{\frac{\alpha}{4},\alpha}} )^{2} \| v-w \|_{C^{\frac{\alpha}{4},\alpha}} $$
\end{enumerate}
with $C=C(n)$.
Similar statements are true for functions in $C^{k,\beta}([0,T])$ and $C^{k,\beta}([0,1])$.
\end{lemma}
\begin{proof}
The main observation is that one has to be  careful about the treatment of the H\"older seminorms.
The first statement relies on the equivalence of the  $l_{2}$-norm and $l_{1}$-norm in $\R^{n}$, which gives
$$| |v(x)|- |v(y)|| \leq |v(x)-v(y)| \leq C(n) \sum_{j=1}^{n} |v^{j}(x) - v^{j}(y)| $$
for any vector valued map $v$.
The last inequality is needed because of our convention for the H\"older norm of a vector valued function.
For the second statement, the aim is to manipulate the
 considered map in such a way that it is written as a product of functions and  we can  apply the previous lemma. We can write
\begin{align*}
\| \,|v|-|w|  \, \|_{C^{\frac{\alpha}{4},\alpha}} &=  \left  \| \frac{|v|^{2} -|w|^{2}}{|v|+|w|}\right \|_{C^{\frac{\alpha}{4},\alpha}} \leq C \| |v|^{2}-|w|^{2}\|_{C^{\frac{\alpha}{4},\alpha}}
\left \| \frac{1}{|v|+|w|} \right \|_{C^{\frac{\alpha}{4},\alpha}}\\
& \leq C \left \| \frac{1}{|v|+|w|} \right\|^{2}_{C^{0}} \|   |v|+|w| \|_{C^{\frac{\alpha}{4},\alpha}}
\left \|  \sum_{j=1}^{n}( (v^{j})^{2} -(w^{j})^{2})  \right\|_{_{C^{\frac{\alpha}{4},\alpha}}}\\
& \leq C \sum_{j=1}^{n} \left \| \frac{1}{|v|+|w|} \right\|^{2}_{C^{0}} \|   |v|+|w| \|_{C^{\frac{\alpha}{4},\alpha}}
\left \| v^{j}+w^{j}  \right\|_{_{C^{\frac{\alpha}{4},\alpha}}}  \left \| v^{j}-w^{j}  \right\|_{_{C^{\frac{\alpha}{4},\alpha}}} 
\end{align*}
and the claim follows.
\end{proof}

We will also need that the composition of H\"older functions is again H\"older. Here one needs to pay attention since, in general, if $f \in C^{0,\alpha}(I), g\in C^{0,\beta}(I)$ then $f \circ g \in C^{0,\alpha \beta}(I)$, i.e. the H\"older exponent of the composition is given by the product of the H\"older exponents. Therefore, in order not to lose in regularity by applying directly this rule, we need to look carefully at the terms we are working with. In particular we exploit that we always consider the convolution of a H\"older map with a diffeomorphism and hence we do not lose in the H\"older power.
\begin{rem}\label{rem:reparid}
If $f_0\in C^{k,\alpha}([0,1])$, $k \geq 4$, then the diffeomorphism $\phi$ defined as in \eqref{eq:repar} is also in $C^{k,\alpha}([0,1])$ by the previous lemmata. Then $(\pa_x^i f_0)\circ \phi \in C^{0,\alpha}([0,1])$ for $0\leq i \leq k$ thanks to the fact that $\phi$ is a diffeomorphism and hence in particular in $C^{0,1}([0,1])$. Since
$\pa_x^k (f_0 \circ \phi)$ is a polynomial in the maps $(\pa_x^i f_0)\circ \phi$ and (several products of) $\pa_x^j \phi$, for $1 \leq i \leq k$, $1 \leq j\leq k$, we see that $ f_0 \circ \phi \in C^{k,\alpha}([0,1])$. 
\end{rem}

\begin{lemma}\label{lem:propHolder}
Let $T<1$ and $v \in C^{ \frac{4+\alpha}{4}, 4+\alpha}([0,T]\times[0,1])$ such that $v(0,x)= 0$, for any $x \in [0,1]$ then
$$ \| \pa_x^l v \|_{C^{\frac{m+\alpha}{4}, m+\alpha}} \leq C(m) T^{\beta} \|  v \|_{C^{\frac{4+\alpha}{4}, 4+\alpha}} $$
for all $l,m \in \mathbb{N}_0$ such that $l+m<4$. Here $\beta=\max \{ \frac{1-\alpha}{4}, \frac{\alpha}{4} \} \in (0,1)$; more precisely for $l\geq 1$ then $\beta=\frac{\alpha}{4}$.

In particular, for each $x \in [0,1]$ fixed
$$ \| \pa_x^l v (\cdot, x) \|_{C^{0,\frac{m+\alpha}{4}}([0,T])} \leq C(m) T^{\beta} \|  v \|_{C^{\frac{4+\alpha}{4}, 4+\alpha}} $$
for all $l,m \in \mathbb{N}_0$ such that $l+m<4$.
\end{lemma}
\begin{proof}
Since  by definition of the norm we have (recall $m,l <4$)
\begin{align}\label{eq:normapp1}
\|\partial_{x}^{l} v \|_{C^{ \frac{m+\alpha}{4}, m+\alpha}} =
\sum_{j=0}^{m} \sup_{[0,T] \times [0,1]} |\partial_{x}^{j+l}v (t,x)| +  [\partial_{x}^{m+l}v ]_{\alpha,x} + \sum_{j=0}^{m}  [\partial_{x}^{j} \partial_{x}^{l} v ]_{\frac{m+\alpha -j}{4},t}
\end{align}
 we observe that for $j \in \{ 0, \ldots, m \}$, $0<l+j \leq l+m  < 4$, and we have
\begin{align*}
\sup_{(t,x) \in [0,T] \times [0,1]} |\partial_{x}^{j+l}v (t,x)| &= \sup_{(t,x) \in [0,T] \times [0,1]} \frac{|\partial_{x}^{j+l}v (t,x) -\partial_{x}^{j+l}v (0,x) |}{|t-0|^{\frac{4+\alpha -(l+j)}{4}}}|t-0|^{\frac{4+\alpha -(l+j)}{4}}\\
&\leq [\partial_{x}^{j+l} v ]_{\frac{4+\alpha -(l+j)}{4},t} \, T^{\frac{4+\alpha -(l+j)}{4}} \leq [\partial_{x}^{j+l} v ]_{\frac{4+\alpha -(l+j)}{4},t} \, T^{\frac{\alpha }{4}}.
\end{align*}
Instead, for the case $j=l=0$ we compute
\begin{align*}
\sup_{(t,x) \in [0,T] \times [0,1]} |v (t,x)| = \sup_{(t,x) \in [0,T] \times [0,1]}  \frac{|v(t,x) -v(0,x)|}{|t-0|} |t| \leq 
T  \left( \sup_{ [0,T] \times [0,1]} |\partial_{t}v| \right) .
\end{align*}

Next, with similar ideas, using the fact that $|x-y| \leq 1$ and $ l+m +1\leq4$ we compute
\begin{align*}
 [\partial_{x}^{m+l}v ]_{\alpha,x} &=\sup_{(t,x), (t,y) \in [0,T] \times [0,1]}\frac{|\partial_{x}^{m+l}v(t,x)-\partial_{x}^{m+l}v(t,y) |}{|x-y|^{\alpha}} \\
 &=\sup_{(t,x), (t,y) \in [0,T] \times [0,1]}\frac{|\partial_{x}^{m+l}v(t,x)-\partial_{x}^{m+l}v(t,y) |}{|x-y|}|x-y|^{1-\alpha}\\
 & \leq \sup_{(t,x) \in [0,T] \times [0,1]} |\partial_{x}^{m+l+1}v(t,x)|\\
 & \leq [\partial_{x}^{m+l+1} v ]_{\frac{4+\alpha -(m+l+1)}{4},t} \, T^{\frac{4+\alpha -(m+l+1)}{4}} 
  \leq [\partial_{x}^{m+l+1} v ]_{\frac{4+\alpha -(m+l+1)}{4},t} \, T^{\frac{\alpha}{4}} 
\end{align*}
and for $ j\in \{0, \ldots, m \}$
\begin{align*}
 [\partial_{x}^{j} \partial_{x}^{l} v ]_{\frac{m+\alpha -j}{4},t}
&=  \sup_{(t,x), (t',x) \in [0,T] \times [0,1]}  \frac{ | \partial_{x}^{j+l} v (t,x) - \partial_{x}^{j+l} v (t',x)|}{|t-t'|^{\frac{m+\alpha +1 -j}{4}}} |t-t'|^{\frac{1}{4}}\\
& \leq [\partial_{x}^{l+j} v ]_{\frac{m+l+1+\alpha -(j+l)}{4},t}\, T^{\frac{1}{4}}.
\end{align*}
The above computation makes sense except for the case $m=3$, $j=0$ (and hence $l=0$), for which $\frac{m+\alpha+1 -j}{4} >1$.
The case $m=3$, $l=j=0$ is treated as follows 
\begin{align*}
[ v ]_{\frac{3+\alpha }{4},t} &= \sup_{(t,x), (t',x) \in [0,T] \times [0,1]}  \frac{|v(t)-v(t')|}{|t-t'|^{\frac{3+\alpha}{4} +\frac{1-\alpha}{4}}} |t-t'|^{\frac{1-\alpha}{4}} \\
& \leq  \left( \sup_{  [0,T] \times [0,1]} |\partial_{t}v| \right) T^{\frac{1-\alpha}{4}} .
\end{align*}
Finally putting all estimates together and recalling that $l+j \leq l+m \leq 3$, $T<1$ and $\alpha <1$ we obtain
(here for $m<3$, for $m=3$ the arguments are similar)
\begin{align*}
&\|\partial_{x}^{l} v \|_{C^{ \frac{m+\alpha}{4}, m+\alpha}}\\
 &\leq \Big (  \sum_{l\neq0 , j=0 }^{m} [\partial_{x}^{j+l} v ]_{\frac{4+\alpha -(l+j)}{4},t}+ [\partial_{x}^{m+l+1} v ]_{\frac{4+\alpha -(m+l+1)}{4},t} + \sum_{ j=0  }^{m}  [\partial_{x}^{l+j} v ]_{\frac{m+l+1+\alpha -(j+l)}{4},t}
\Big) T^{\frac{\alpha}{4}}\\
& \quad + C\left( \sup_{ [0,T] \times [0,1]} |\partial_{t}v| \right) T^{\frac{1-\alpha}{4}} .
\end{align*}
and  the first estimate follows. The second part of the claim follows from the observation that
\begin{equation}\label{eq:estHolbdy}
\| \pa_x^l v (\cdot, x) \|_{C^{0,\frac{m+\alpha}{4}}([0,T])} 
\leq \| \pa_x^l v \|_{C^{\frac{m+\alpha}{4}, m+\alpha}([0,T]\times[0,1])} \, , 
\end{equation}
due to \eqref{eq:normapp1}.
\end{proof}

\bibliography{ref2}
\bibliographystyle{acm}

\end{document}